\documentclass[10pt]{article}
\usepackage{amsmath, amssymb, amsfonts, amsthm, mathtools, xfrac, enumitem}
\usepackage[numbers]{natbib}
\usepackage{epstopdf, graphicx}
\usepackage{bbold} 

\usepackage{xr-hyper}  
\usepackage{hyperref}

\newtheorem{corollary}{Corollary}[section]
\newtheorem{lemma}{Lemma}[section]
\newtheorem{proposition}{Proposition}[section] 
\newtheorem{definition}{Definition}[section]

\makeatletter

\newcommand{\N}{\mathbb{N}}
\newcommand{\R}{\mathbb{R}}

\newtheorem{remark}{Remark}[section]
\newtheorem{thm}{Theorem}

\title{Prescribing Morse scalar curvatures: subcritical blowing-up solutions} 
 
\author
{Andrea Malchiodi \& Martin Mayer \\ \ \\
Published on Journal of Differential Equations \\
\small Scuola Normale Superiore, Pisa, ITALY \\
\small andrea.malchiodi@sns.it , martin.mayer@sns.it
}  

\begin{document}

\maketitle

\begin{abstract}
\noindent Prescribing conformally the scalar curvature of a Riemannian manifold as a given function 
 consists in 
solving an elliptic PDE involving the critical Sobolev exponent. One way of attacking this 
problem consist in using subcritical approximations for the equation, gaining compactness  
properties.  
Together with the results in \citep{MM1} we completely describe the blow-up 
phenomenon in case of uniformly bounded energy, zero weak limit and positive Yamabe invariant. 
In particular for dimension greater or equal to five and Morse functions with non-zero Laplacian 
at each critical point we show, that subsets of critical points with negative Laplacian 
are in one-to-one correspondence with such subcritical blowing-up solutions. 
\end{abstract}

\maketitle

\tableofcontents

 \section{Introduction}

Consider a compact manifold $ (M^n, g_{0}) $ with $ n \geq 3 $ and a conformal metric 
 $$ g = u^{\frac{4}{n-2}} g_{0}, \;u > 0. $$ With this notation the scalar curvature transforms 
via
 $$ 
 R_{g_u} u^{\frac{n+2}{n-2}} = L_{g_0} u = - c_{n} \Delta_{g_{0}} u + R_{g_{0}} u, 
 \quad 
c_n = \frac{4(n-1)}{(n-2)}
 $$ 
with $ \Delta_{g_{0}} $ denoting the Laplace-Beltrami operator of $ g_0 $, cf. \cite{aul}. $ L_{g_0} $ is called 
the {\em conformal Laplacian} and transforms according to 
 $$ 
L_g ( \phi) 
= 
u^{-\frac{n+2}{n-2}} L_{g_0}(u\phi)
. $$ 
In the 70 , Kazdan and Warner considered in \cite{kw} the problem of prescribing the 
scalar curvature of manifolds via conformal deformation of the metric, see also \cite{kw1}, \cite{kw2}. 
By the above transformation law, if one wishes to prescribe $ R_g $ as a 
given function $ K $ then would need to solve
\begin{equation}\label{eq_scin}
L_{g_0} u = K u^{\frac{n+2}{n-2}}\quad \hbox{ on } \quad (M, g_0). 
\end{equation}
There are rather easy obstructions to the solvability of \eqref{eq_scin}. For example, if the 
sign of $ K $ is constant, it has to coincide with that of the first eigenvalue of $ L_{g_0} $. 
Depending on the latter sign, which is conformally invariant, a conformal class of metrics is 
said to be of \emph{negative, zero or positive Yamabe class}. We will discuss for simplicity the 
case of function $ K $ with constant sign, despite in the literature there are many interesting 
papers dealing with changing-sign functions. 

In \cite{kw} Kazdan and Warner proved existence results for zero or negative 
Yamabe classes using the sub- and super-solution method. For positive Yamabe class instead, 
they found a now well-known obstruction to existence on the sphere, namely 
that if $ u $ solves\eqref{eq_scin}, then for $ f $ affine on $ S^n $ one must have
\begin{equation}
\label{eq_kw}
\int_{S^{n}}\langle \nabla K, \nabla f \rangle_{g_{S^n}} u^{\frac{2n}{n-2}} \, d\mu_{g_{S^{n}}} = 0, 
\end{equation}
and hence, for 
conformal curvatures $ K $, the function $ \langle \nabla K, \nabla f \rangle_{g_{S^n}} $ must change sign. 

Later on some existence results were found under conditions that would imply {\em topological richness} 
of the sub-levels of $ K $, contrary to the above example. In two dimensions, where \eqref{eq_scin} is 
replaced by an equation in exponential form, J. Moser showed that the problem is solvable on the 
standard sphere if $ K $ is antipodally symmetric. In higher dimensions, existence results under the 
action of symmetry groups were proven in \cite{es} and \cite{hv}, \cite{hv93}. 
\smallskip
A general difficulty in studying \eqref{eq_scin} is the lack of compactness due to the presence 
of the critical exponent. A typical phenomenon encountered here is that of {\em bubbling}. 
{\em Bubbles} are solutions of \eqref{eq_scin} on $ S^{n} $ with $ K\equiv 1 $ and these 
 arise as profiles of general diverging solutions and 
 were classified in \cite{cgs}, see also \cite{auw}, \cite{tal}. From the variational point of 
 view bubbles generate diverging Palais-Smale sequences for the Euler-Lagrange energy of \eqref{eq_scin}
 $$ 
 J(u)
 =
 J_{K}(u) 
 = 
 \frac
 {\int_{M} \left( c_{n}\vert \nabla u\vert_{g_{0}}^{2}+R_{g_{0}}u^{2} \right) d\mu_{g_{0}}}
 {(\int_{M} Ku^{\frac{2n}{n-2}}d\mu_{g_{0}})^{\frac{n-2}{n}}}. 
 $$ 
As seen from a formal expansion of $ J $ on a finite sum of bubbles, cf. the introduction in \cite{MM1}, 
 the mutual interaction among bubbles becomes weaker as $ n $ increases. As a consequence in case $ n=3 $ 
 at most one bubble can form.
Exploiting this and after some work on $ S^2 $ by A. Chang and P. Yang in \cite{[ACPY1]}, \cite{[ACPY2]}, 
A. Bahri and J.M. Coron proved an existence result in \cite{bc} on $ S^3 $ assuming 
that $ K $ is a Morse function and 
\begin{equation}\label{eq:nd}
\{ \nabla K= 0 \} \cap \{ \Delta K = 0 \} = \emptyset; 
\end{equation}
\begin{equation}\label{eq:bcin}
\sum_{\{ x \in M\; : \; \nabla K(x) = 0, \Delta K(x) < 0\}}(-1)^{m(x, K)} \neq -1, 
\end{equation}
where $ m(x, K) $ denotes the Morse index of $ K $ at $ x $, cf.
\cite{[ACGPY]} and \cite{[SZ]} for more general related results. The above existence statement was extended to 
arbitrary dimensions in 
 \cite{yy} for functions satisfying a suitable flatness condition, and in \cite{[ACPY3]}, \cite{agp}, \cite{mal} for functions $ K $ close to 
 a positive constant in the $ C^2 $-sense.

 In four dimensions, see \cite{[BCCH4]} and \cite{yy2}, it was shown that even if multiple bubbles can form, they cannot be too close to each-other; 
such phenomenon is usually refereed to as {\em isolated simple blow-up}.
 Results of different kind were also proven in \cite{[CD]} for $ n = 2 $ 
and in \cite{[BI]}\cite{[BCCHhigh]}, \cite{[BIEG]}, see also Chapter 6 in \cite{aul}.

Two main approaches have been used to understand the blow-up phenomenon, 
namely sub-critical approximations or the construction of pseudo-gradient flows.
In this paper we focus on the former, while the other one will be the subject of 
\cite{MM3}, where a one-to-one correspondence of zero weak limit blowing-up 
solutions with bounded energy and {\em critical points 
at infinity} is shown, see also \cite{MM5}.
Consider the problem 
\begin{equation}\label{eq:scin-tau}
- c_{n} \Delta_{g_{0}} u + R_{g_{0}} u = K \, u^{\frac{n+2}{n-2}-\tau}, \; \quad
0<\tau \ll 1, 
\end{equation}
 which upon rescaling is the Euler-Lagrange equation for the functional
\begin{equation}\label{eq:JKt}
J_{\tau}(u)
=
\frac
{\int_{M} \left( c_{n}\vert \nabla u\vert_{g_{0}}^{2}+R_{g_{0}}u^{2} \right) d\mu_{g_{0}}}
{(\int_M Ku^{p+1}d\mu_{g_{0}})^{\frac{2}{p+1}}}, 
 \quad u \in \mathcal{A}.
\end{equation}
Being now the exponent lower than critical, solutions can be easily found, even though one could lose 
uniform estimates as $ \tau $ tends to zero. In 
\cite{[ACGPY]}, \cite{[SZ]}, \cite{yy} the single-bubbling behaviour for diverging solutions 
of \eqref{eq:scin-tau} was proved. Then by degree- or Morse-theoretical arguments it was 
shown that under \eqref{eq:bcin} there must be families of solutions that stay uniformly bounded, 
therefore converging to solutions of \eqref{eq_scin}. For this argument to work, one crucial step 
was to completely characterize blowing-up solutions of \eqref{eq:scin-tau}, showing that in three dimensions single 
blow-ups occur at any critical point of $ K $ with negative Laplacian and that they are unique. 
On four-dimensional spheres, a similar property was proved in \cite{yy2}for multiple blow-ups, see also \cite{[BCCH4]}, 
assuming a suitable condition related to the multi-bubble interactions.

For Morse functions, if $ n \geq 5 $ the situation is more involved, and blow-ups might be 
possibly of infinite energy, see e.g. \cite{cl1}, \cite{cl2}, \cite{cl3}, \cite{wy}. In \cite{MM1} it was however proved that 
if a sequence of blowing-up solutions has uniformly-bounded $ W^{1, 2} $-energy and 
zero weak limit, then blow-ups are still isolated simple. Although the result is similar to the 
case of dimensions three and four, the phenomenon is somehow opposite since it is 
{\em driven} by the function $ K $ rather than from the mutual bubble interactions. Both assumptions, i.e. 
zero weak limit and bounded energy, are indeed natural. If the former fails then problem \eqref{eq_scin} 
would have a solution; the second one instead is usually found when using min-max or Morse-theoretical 
arguments, as it will be done in \cite{MM4}. 
However, differently from $ n = 3, 4 $, in \cite{MM1} no restriction is proven on the number or 
location of blow-up points, provided they occur at critical points of $ K $ with negative Laplacian. 

\medskip
 
In this paper reshow, that the characterization of the above blow-ups in \cite{MM1} is 
sharp, namely that they can occur at arbitrary subsets of 
 $$ \{ \nabla K = 0\} \cap \{ \Delta K < 0\}. $$ 
Furthermore, we prove uniqueness of such solutions, their non-degeneracy and determine 
their Morse index. Our main result is the following one, that follows from
Theorem 1 in \cite{MM1} and from Proposition \ref{prop_subcritical_existence}, 
 Corollary \ref{cor_restricted_second_variation}. 

\begin{thm}\label{t:ex-multi}
Let $ (M, g) $ be a compact manifold of dimension $ n \geq 5 $ of positive Yamabe invariant and 
let $ K : M \longrightarrow \R $ be a positive Morse function satisfying \eqref{eq:nd}. 
Let $ x_1, \dots, x_q $ be distinct critical points of $ K $ with negative Laplacian. Then there exists, as $ \tau \longrightarrow 0 $ and up to scaling, 
a unique solution $ u_{\tau, x_1, \dots, x_q} $ developing a simple bubble 
at each point $ x_i $ converging weakly to zero in $ W^{1, 2}(M, g) $ as $ \tau \longrightarrow 0 $. 
Moreover and up to scaling $ u_{\tau, x_1, \dots, x_q} $ is non-degenerate for $ J_\tau $ 
and 
 $$ m(J_\tau, u_{\tau, x_1, \dots, x_q}) = (q-1) + \sum_{i=1}^q 
(n-m(K, x_i)). $$ 
Conversely all blow-ups of uniformly bounded energy and zero weak limit type
are as above. 
\end{thm}

\medskip

\noindent As it will be shown in \cite{MM4}, for $ n \geq 5 $ there cannot be a direct counterpart of
\eqref{eq:bcin}, which is an index-counting condition. However, existence results 
of different type will be derived there. 

\medskip

 \begin{remark}\label{r:precise} 
 \begin{enumerate}[label=(\roman*)]
\item 
More precise expressions for $ u_{\tau, x_1, \dots, x_q} $ are given by
 \begin{equation*}
\bigg\| u_{m} - \sum_{j=1}^q \alpha_{j, m} \delta_{\lambda_{j, m}, a_{j, m}} \bigg\|_{W^{1, 2}(M, g_0)}
\longrightarrow 0 \quad \text{ as } \quad m \longrightarrow\infty, 
\end{equation*} 
and
 $$ 
\alpha_{j, m}=\frac{\Theta}{K(x_{j})^{\frac{n-2}{4}}}+o(1)
, \quad a_{j, m} \longrightarrow x_j
\quad \text{ and }\quad 
\lambda_{j, m} \simeq \lambda_{\tau_{m}} = \tau^{- \frac 12}_{m}.
 $$ 
Here the multiplicative constant $ \Theta $ depends on the blowing-up 
solutions but it is independent of $ j $. 
For this and more precise formulae we refer to 
 Section \ref{s:ex} and Theorem \ref{lem_top_down_cascade} 
in the Appendix. If $ n = 4 $, the same conclusions hold replacing $ \Delta K (a_j) < 0 $ for all $ j $ with (iv) of Theorem 2 in \cite{MM1}. 
\item 
Although upon scaling the above solutions $ u_{\tau, x_1, \dots, x_q} $ 
are non-degenerate, they Hessian of $ J_\tau $ there has $ \sum_{i=1}^q (n-m(K, x_i)) $ 
eigenvalues approaching zero as $ \tau \longrightarrow 0 $, see Section \ref{s:2nd}. 
\item 
Theorem \ref{t:ex-multi} gives a one-to-one correspondence of zero weak limit subcritical blow-up solutions to subsets of critical points of $ K $ with negative Laplacian, while in \cite{MM3} this correspondence is shown with zero weak limits, i.e. pure critical points at infinity of energy decreasing type, cf. \cite{bab}, \cite{may-cv}.
\end{enumerate}
\end{remark}

\noindent The proof of Theorem \ref{t:ex-multi} relies on the estimates in \cite{MM1} and 
a finite dimensional reduction, see e.g. \cite{am}, with a careful asymptotic 
analysis. In dimension four this approach was used in Section 2 of \cite{yy2}.
Here we show that in higher dimensions blow-up might occur at arbitrary 
critical points of $ K $ with negative Laplacian, which affects the global structure 
of the solutions of problem \eqref{eq_scin}. Via careful expansions, we also 
determine the Hessian of the Euler-Lagrange functional and the Morse index of these 
solutions, which we prove to be non-degenerate.

The solutions we consider here lie in a set $ V(q, \varepsilon)\subset W^{1, 2}(M, g_0) $, which contains a manifold of approximate 
solutions for \eqref{eq:scin-tau}, namely 
 $$ \sum_{i=1}^{q} \alpha^i \varphi_{a_{i}, \lambda_{i}}, $$ and is transversally non-degenerate, Section \ref{s:prel} for notation. This allows 
to solve \eqref{eq:scin-tau} orthogonally to this manifold via a proper transversal correction to 
the approximate solutions, see Definition \ref{d:barv} and Lemma \ref{l:v-bar}, 
and reduce to the study of the tangent component. By Theorem \ref{lem_top_down_cascade} from \cite{MM1}
we can reduce ourselves to a smaller set $ \bar V(q, \varepsilon) $, see \eqref{def_refined_neighbourhood}, 
where more precise estimates hold for the gradient of $ J_\tau $. These allow us to 
use an orthogonal correction $ \bar v $ small in size, solve also for the tangent component and to estimate the second differential of 
 $ J_\tau $ at $ \sum_{i=1}^{q} \alpha^i \varphi_{a_{i}, \lambda_{i}} + \bar v $, see Section \ref{s:2nd}. 
Finally this allows in turn to compute the Morse index of the solutions $ u_{\tau, x_1, \dots, x_q} $ 
and to prove their uniqueness. In this step we show that even though the correction $ \bar v $ is of the 
same order as some eigenvalues of $ \partial ^{2}J_\tau $ some cancellation occur in the corresponding
estimates.

\

The plan of the paper is the following. In Section \ref{s:prel} we collect some preliminary material 
concerning approximate solutions and the finite-dimensional reduction of the problem, which is 
then worked-out in detail in Section \ref{s:ex}. In Section \ref{s:2nd} we study the Hessian 
of the Euler-Lagrange functional $ J_\tau $ in $ \bar V(q, \varepsilon) $, finding a proper base 
with respect to which the Hessian nearly diagonalizes. Finally we collect in an Appendix some useful and 
technical estimates from \cite{MM1} and a table of constants.

 \

\noindent {\bf Acknowledgements.}
A.M. has been supported by the project {\em Geometric Variational Problems} and {\em Finanziamento a supporto della ricerca di base} from Scuola Normale Superiore and by MIUR Bando PRIN 2015 2015KB9WPT $ _{001} $. He is also member of GNAMPA as part of INdAM.

\section{Preliminaries} \label{s:prel}

In this section we collect some background and preliminary material, concerning the 
variational properties of the problem and some estimates on highly-concentrated 
approximate solutions of bubble type.

We consider a smooth, closed Riemannian manifold 
 $ 
M=(M^{n}, g_{0}) 
 $ 
with volume measure $ \mu_{g_{0}} $ and scalar curvature $ R_{g_{0}} $.
Letting 
 $$ 
\mathcal{A}=
\{
u\in W^{1, 2}(M, g_{0})\mid u\geq 0, u\not \equiv 0
\}
 $$ 
the {\em Yamabe invariant}is defined as 
\begin{equation*}\begin{split}
Y(M, g_{0})
= &
\inf_{\mathcal{A}}
\frac
{\int \left( c_{n}\vert \nabla u \vert_{g_{0}}^{2}+R_{g_{0}}u^{2} \right) d\mu_{g_{0}}}
{(\int u^{\frac{2n}{n-2}}d\mu_{g_{0}})^{\frac{n-2}{n}}}
\quad \text{ with } \quad 
c_{n}=4\frac{n-1}{n-2}
\end{split}\end{equation*}
and it turns out to depend only on the conformal class of $ g_0 $. 
We will assume that this invariant is positive, 
i.e. $ (M, g_0) $ to be of \emph{positive Yamabe class}. 
As a consequence the {\em conformal Laplacian} 
\begin{equation*}\begin{split} 
L_{g_{0}}=-c_{n}\Delta_{g_{0}}+R_{g_{0}}
\end{split}\end{equation*}
 is a positive and self appointed operator. Without loss of generality 
we assume $ R_{g_{0}}>0 $ and denote by
 $$ 
G_{g_{0}}:M\times M \setminus \Delta \longrightarrow\R_{+}
 $$ 
the Green  function of $ L_{g_0} $. 
Considering a conformal metric $$ g=g_{u}=u^{\frac{4}{n-2}}g_{0} $$ 
there holds 
\begin{equation*}\begin{split} d\mu_{g_{u}}=u^{\frac{2n}{n-2}}d\mu_{g_{0}}
\quad \text{ and } \quad 
R=R_{g_{u}}=u^{-\frac{n+2}{n-2}}(-c_{n} \Delta_{g_{0}} u+R_{g_{0}}u) =
u^{-\frac{n+2}{n-2}}L_{g_{0}}u.
\end{split}\end{equation*}
Note that 
\begin{align*}
c\Vert u \Vert_{W^{1, 2}(M, g_0)}
\leq
\int u \, L_{g_{0}}u \, d\mu_{g_{0}}
\leq 
C\Vert u \Vert_{W^{1, 2}(M, g_0)}.
\end{align*}
In particular we may define 
 $$ \Vert u \Vert^2 = \Vert u \Vert_{L_{g_0}}^2 = \int u \, L_{g_{0}}u \, d\mu_{g_{0}} $$ 
and use $ \Vert \cdot \Vert $ as an equivalent norm on $ W^{1, 2}(M, g_0) $. 
Setting
 $$ R=R_{u} \quad \text{ for } \quad g=g_{u}=u^{\frac{4}{n-2}}g_{0} $$ we have 
\begin{equation}\label{eq:r}
r=r_{u}=\int R d\mu_{g_{u}}=\int u L_{g_{0}} ud\mu_{g_{0}}, 
\end{equation}
and hence 
\begin{equation}\label{eq:kp}
J_{\tau}(u)=\frac{r}{k_{\tau}^{\frac{2}{p+1}}}
\; \quad \text{ with } \; \quad k_{\tau} = \int K \, u^{p+1} d \mu_{g_{0}}. 
\end{equation}
The first- and second-order derivatives of the functional $ J_\tau $ are given by 
\begin{equation}\label{first_variation_evaluating} 
\partial J_{\tau}(u)v
= 
\frac{2}{k_{\tau}^{\frac{2}{p+1}}}
\big[\int L_{g_{0}}uvd\mu_{g_{0}}-\frac{r}{k_{\tau}}\int Ku^{p}vd\mu_{g_{0}}\big]; 
\end{equation}
and
\begin{equation}\label{second_variation_evaluating} 
\begin{split}
\partial^{2} J_{\tau}(u) vw
= &
\frac{2}{k_{\tau}^{\frac{2}{p+1}}}
\big[\int L_{g_{0}}vwd\mu_{g_{0}}-p\frac{r}{k_{\tau}}\int Ku^{p-1}vwd\mu_{g_{0}}\big] \\
& -
\frac{4}{k_{\tau}^{\frac{2}{p+1}+1}}
\big[
\int L_{g_{0}}uvd\mu_{g_{0}}\int Ku^{p}wd\mu_{g_{0}}
\\ & \quad\quad\quad\quad\quad\quad +
\int L_{g_{0}}uwd\mu_{g_{0}}\int Ku^{p}vd\mu_{g_{0}}
\big] \\
& +
\frac{2(p+3)r}{k_{\tau}^{\frac{2}{p+1}+2}}
\int Ku^{p}vd\mu_{g_{0}}\int Ku^{p}wd\mu_{g_{0}}.
\end{split}
\end{equation}
In particular $ J_{\tau} $ is of class $ C^{2, \alpha}_{loc}(\mathcal{A}) $ and for $ \varepsilon > 0 $ uniformly H\"older continuous on each
set of the form 
\begin{align*}
U_{\varepsilon}=\{u\in\mathcal{A} \mid\varepsilon<\Vert u \Vert, \, J_{\tau}(u)\leq \varepsilon^{-1}\}.
\end{align*}
To understand the blow-up phenomenon, it is convenient to consider 
some highly concentrated approximate solutions to \eqref{eq_scin}. 
Let us first recall the construction of {\em conformal normal coordinates} from \cite{lp}.
Given $ a \in M $ these are defined as geodesic normal coordinates 
for a suitable conformal metric $ g_{a} \in [g_{0}] $. 
Let $ r_a $ be the geodesic distance from $ a $ with respect to the metric $ g_a $. 
With this choice, the expression of the 
Green  function $ G_{g_{ a }} $ for the conformal 
Laplacian $ L_{g_{a}} $ with pole at $ a \in M $, denoted by 
 $ G_{a}=G_{g_{a}}(a, \cdot) $, simplifies considerably. In Section 6 of \cite{lp}
one can find the expansion 
\begin{equation}\begin{split} \label{eq:exp-G}
G_{ a }=\frac{1}{4n(n-1)\omega _{n}}(r^{2-n}_{a}+H_{ a })
, \; 
H_{ a }=H_{r, a }+H_{s, a }\; \text{ for } \; g_{a}=u_{a}^{\frac{4}{n-2}}g_{0}.
\end{split}\end{equation}
Here 
 $ 
r_{a}=d_{g_{a}}(a, \cdot)
 $ 
and
 $ H_{r, a }\in C^{2, \alpha}_{loc} $, while the {\em singular} error term is of type
\begin{equation*}
\begin{split}
H_{s, a}
=
O
\begin{pmatrix}
r_{a}& \text{ for }\, n=5
\\
\ln r_{a} & \text{ for }\, n=6
\\
r_{a}^{6-n} & \text{ for }\, n\geq 7
\end{pmatrix}.
\end{split}
\end{equation*} 
The leading term in $ H_{s, a} $ for $ n=6 $ is 
 $ \mathbb{W} $ the Weyl tensor.
Define
\begin{equation}\label{eq:bubbles}
\varphi_{a, \lambda }
= 
u_{ a }\left(\frac{\lambda}{1+\lambda^{2} \gamma_{n}G^{\frac{2}{2-n}}_{ a }}\right)^{\frac{n-2}{2}}, \quad 
G_{ a }=G_{g_{ a }}( a, \cdot)
\end{equation}
for $ \lambda > 0 $ large, where $ \gamma_{n}=(4n(n-1)\omega _{n})^{\frac{2}{2-n}} $ is chosen so that 
\begin{equation}
\gamma_{n}G^{\frac{2}{2-n}}_{ a }(x) 
= 
d_{g_a}^2(a, x) + o(d_{g_a}^2(a, x))
\; \text{ as } \; 
x \longrightarrow a.
\end{equation}
Such functions are approximate solutions of \eqref{eq_scin}, see Lemma \ref{lem_emergence_of_the_regular_part}, 
and for suitable values of $ \lambda $ depending on $ \tau $ these are also approximate solutions 
of \eqref{eq:scin-tau}, see Lemma \ref{lem_upper_bound} for a multi-bubble version. 

\
\noindent {\bf{Notation.}} For $ p \geq 1 $, $ L^p_{g_0} $ will stand for
the family of functions of class $ L^p $ with 
respect to the measure $ d \mu_{g_{0}} $. 
Recall also that for $ u \in W^{1, 2}(M, g_0) $ we have set $ r_u = \int u L_{g_0} u d \mu_{g_{0}} $, while for 
 $ a \in M $ we denote by $ r_a $ the geodesic distance from $ a $ with respect to the conformal metric $ g_a $ 
introduced before. 
For a finite set of points $ \{a_i\}_i $ of $ M $ we will denote by $ K_{i}, \nabla K_i, W_{i} $ the quantities 
 $ K(a_{i}), \nabla K(a_{i}), |\mathbb{W}(a_i)|^2 $ etc..
\\ \\
\noindent
For $ k, l=1, 2, 3 $ and $ \lambda_{i} >0, \, a _{i}\in M, \, i= 1, \ldots, q $ let
\begin{enumerate}[label=(\roman*)]
 \item \quad 
 $ \varphi_{i}=\varphi_{a_{i}, \lambda_{i}} $ and $ (d_{1, i}, d_{2, i}, d_{3, i})=(1, -\lambda_{i}\partial_{\lambda_{i}}, \frac{1}{\lambda_{i}}\nabla_{a_{i}}); $ 
 \item \quad
 $ \phi_{1, i}=\varphi_{i}, \;\phi_{2, i}=-\lambda_{i} \partial_{\lambda_{i}}\varphi_{i}, \;\phi_{3, i}= \frac{1}{\lambda_{i}} \nabla_{ a _{i}}\varphi_{i} $, so
 $ 
\phi_{k, i}=d_{k, i}\varphi_{i}. 
 $ 
\end{enumerate}
Note, that the $ \phi_{k, i} $ are uniformly bounded in $ W^{1, 2}(M, g_0) $ for any
$ \lambda_{i} >0$.

\

We next recall a standard finite-dimensional 
reduction for functions that are close in $ W^{1, 2} $ to a finite sum of bubbles, wherefore we define
\begin{equation}\begin{split} \label{eq:eijm}
 \varepsilon_{i, j}
 =
 \bigg(
 \frac{\lambda_{j}}{\lambda_{i}}
 +
 \frac{\lambda_{i}}{\lambda_{j}}
 +
 \lambda_{i}\lambda_{j}\gamma_{n}G_{g_{0}}^{\frac{2}{2-n}}(a _{i}, a _{j})
 \bigg)^{\frac{2-n}{2}}.
 \end{split}\end{equation}
 Given $ \varepsilon>0, \; q\in \N, \; u\in W^{1, 2}(M, g_{0}) $ and $ ( \alpha^{i}, \lambda_{i}, a_{i})\in (\R^{q}_{+}, \R^{q}_{+}, M^{q}) $, we set 
\vspace{16pt}
\begin{enumerate}[label=(\roman*)]
\item \quad 
\vspace{-31pt}
\begin{equation*}
\begin{split}
A_{u}(q, \varepsilon)
 =
 \{ 
 ( \alpha^{i}, \lambda_{i}, a_{i}) \mid 
\;
 \underset{i\neq j}{\forall\;}\; 
\lambda_{i}^{-1}, \lambda_{j}^{-1}, 
&
\varepsilon_{i, j}, \bigg\vert 1-\frac{r\alpha_{i}^{\frac{4}{n-2}}K(a_{i})}{4n(n-1)k_{\tau}}\bigg\vert, 
\\ & 
\Vert u-\alpha^{i}\varphi_{a_{i}, 
\lambda_{i}}\Vert
<
\varepsilon, 
\lambda_{i}^\tau 
< 
1 + \varepsilon
 \};
\end{split}
\end{equation*}
\item \quad
 $ 
 V(q, \varepsilon)
 = 
 \{
 u\in W^{1, 2}(M, g_{0})
 \mid
 A_{u}(q, \varepsilon)\neq \emptyset
 \}, 
 $ 
 \end{enumerate}
 see\eqref{eq:r}, \eqref{eq:kp} and \eqref{eq:bubbles}. For 
 $ A_{u}(q, \varepsilon) $ to be non-empty we 
 will always assume that $ \tau \ll \varepsilon $. 
Under the above conditions on the parameters $ \alpha_{i}, a_i $ and $ \lambda_{i} $ the functions 
 $ \sum_{i=1}^q \alpha^{i}\varphi_{a_{i}, \lambda_{i}} $ constitute a smooth manifold in $ W^{1, 2}(M, g_0) $, 
which implies the following well known result, cf. \cite{bab}.

 \begin{proposition} 
 \label{prop_optimal_choice} 
 Given $ \varepsilon_{0}>0 $ there exists $ \varepsilon_{1}>0 $ such  that for $ u\in V( q, \varepsilon) $ 
 with $ \varepsilon<\varepsilon_{1} $, the problem 
 \begin{equation*}\begin{split}
 \inf_
 {
 (\tilde\alpha_{i}, \tilde a_{i}, \tilde\lambda_{i})\in A_{u}(q, 2\varepsilon_{0}) 
 }
 \int 
 (
 u
 -
 \tilde\alpha^{i}\varphi_{\tilde a_{i}, \tilde \lambda_{i}}
 ) L_{g_{0}}
 (
 u
 -
 \tilde\alpha^{i}\varphi_{\tilde a_{i}, \tilde \lambda_{i}}
 )
 d\mu_{g_{0}}
 \end{split}\end{equation*}
 admits an unique minimizer $ (\alpha_{i}, a_{i}, \lambda_{i})\in A_{u}(q, \varepsilon_{0}) $ and we set 
 \begin{equation}\begin{split} \label{eq:v}
 \varphi_{i}=\varphi_{a_{i}, \lambda_{i}}, \quad v=u-\alpha^{i}\varphi_{i}, \quad K_{i}=K(a_{i}).
 \end{split}\end{equation}
 Moreover $ (\alpha_{i}, a_{i}, \lambda_{i}) $ depends smoothly on $ u $.
 \end{proposition} 
 
 \medskip
 
\noindent
 The term
 $ v=u-\alpha^{i}\varphi_{i} $ is orthogonal to all
 $ 
 \varphi_{i}, -\lambda_{i}\partial_{\lambda_{i}}\varphi_{i}, \frac{1}{\lambda_{i}}\nabla_{a_{i}}\varphi_{i}, 
 $ with respect to the product 
 \begin{equation*}\begin{split}
 \langle \cdot, \cdot\rangle_{L_{g_{0}}}
 =
 \langle L_{g_{0}}\cdot, \cdot\rangle_{L^{2}_{g_{0}}}. 
 \end{split}\end{equation*}
Finally for $ u\in V( q, \varepsilon) $ let
 \begin{equation} \begin{split} 
 \label{eq:Hu}
 H_{u} = H_{u}( q, \varepsilon)
 =
 \langle 
\varphi_{i}, \lambda_{i}\partial_{\lambda_{i}}\varphi_{i}, \frac{1}{\lambda_{i}}\nabla_{a_{i}}\varphi_{i}
 \rangle
 ^{\perp_{L_{g_{0}}}}.
 \end{split}
 \end{equation}

\section{Existence of subcritical solutions}\label{s:ex}

Theorem\ref{lem_top_down_cascade}, from \cite{MM1}, 
describes in detail the behaviour as $ \tau \longrightarrow 0 $ of blowing-up solutions to \eqref{eq:scin-tau} with 
uniformly bounded energy and zero weak limit in $ V(q, \varepsilon) $, 
providing positive lower bounds on $ \| \partial J_\tau \| $ in a suitable subset of the functional space. 
In view of this, we can restrict our attention
to {\em centers} $ a_{1}, \ldots, a_{q} $ close to \underline{distinct} critical points $ x_1, \dots, x_q $ of $ K $ with negative Laplacian. More precisely
for $ n \geq 6 $ we can assume the  following conditions, which for $ n = 5 $ are slightly modified, 
\begin{enumerate}[label=(\roman*)]
\item \quad $ 
\vert \alpha_{j}-\Theta \sqrt[p-1]{\frac{\lambda_{j}^{\theta}}{K(a_{j})}}\vert <\frac{\epsilon}{\lambda^{3}}; $ 
 \item \quad 
 $ 
 \vert 
 \frac{\bar a_{j}}{\lambda_{j}}
 +
 c_{1}(\nabla^{2} K(x_{j}))^{-}\frac{\nabla \Delta K(x_{j})}{\lambda_{j}^{3}}
 \vert
 \leq \frac{\epsilon}{\lambda^{3}};
 $ 
 \item \quad 
 $ 
 \vert 
 \lambda_{j}^{2}
 +
 c_{2}\frac{\Delta K(x_{j})}{K(x_{j})\tau}
 \vert 
 \leq \frac{\epsilon}{\lambda^{2}}, 
 $ 
\end{enumerate}
 for $ \lambda^{2}=\frac{1}{\tau} $ and some 
 $$ x_{j}\in \{\nabla K=0\}\cap \{\Delta K<0\}
\quad \text{ with } \quad 
x_{i}\neq x_{j} \quad \text{ for } \quad i\neq j. $$ Here
 $ \Theta > 0 $ is uniformly bounded and bounded away from zero and depends on 
 the function in $ V(q, \varepsilon) $, determined in Remark 6.2 of \cite{MM1}. 
We then define a neighbourhood of potential subcritical blowing-up solutions as
\begin{equation}\label{def_refined_neighbourhood}
\bar V(q, \varepsilon)
=
\{
u\in V(q, \varepsilon)\mid \; \text{ (i), (ii) and (iii) above hold true}
\}.
\end{equation}
Indeed from Lemmata \ref{lem_alpha_derivatives_at_infinity}, \ref{lem_lambda_derivatives_at_infinity} and \ref{lem_a_derivatives_at_infinity}it follows, that there 
exists $ \tilde{\epsilon} > 0 $, tending to zero as $ \varepsilon \longrightarrow 0 $, such that 
\begin{equation*}
\vert \partial J_{\tau}(u) \vert \gtrsim \frac{\tilde \epsilon}{\lambda^{3}}\; \text{ for }\;u\in V(q, \varepsilon)\setminus \bar V(q, \varepsilon)\; \text{ with }\; k_{\tau}=1, 
\end{equation*}
so this justifies to look for solutions in $ \bar V(q, \varepsilon) $ only. 
Moreover for 
 $$ \alpha^{i}\varphi_{i}\in \bar V(q, \varepsilon)
\quad \text{ with } \quad 
c<\alpha_{i}<C $$ 
we have the expansion 
\begin{equation}
J_{\tau}(\alpha^{i}\varphi_{i}+v)
=
J_{\tau}(\alpha^{i}\varphi_{i})
+
\partial J_{\tau}(\alpha^{i}\varphi_{i})v
+
\frac{1}{2}\partial^{2}J_{\tau}(\alpha^{i}\varphi_{i})v^{2}
+
O(\Vert v \Vert^{3}). 
\end{equation} 
Recall the uniform positivity of $ \partial^{2}J_{\tau}(\alpha^{i}\varphi_{i}) $ on $ H_{u}(q, \varepsilon) $, 
cf. \eqref{eq:Hu} and \cite{bab}, which justifies the following 
\begin{definition}\label{d:barv}
For $ 
\alpha^{i}\varphi_{i}\in V(q, \varepsilon)
 $
we define $ \bar v $ as the unique solution of the minimization problem
\begin{equation}\label{bar_v_minimization_problem}
J_{\tau}(\alpha^{i}\varphi_{i}+\bar v)=\min_{v\in H_{\alpha^i \varphi_{i}}, \Vert v \Vert <\varepsilon}J_{\tau}(\alpha^{i}\varphi_{i}+v). 
\end{equation}
\end{definition}

\begin{lemma}\label{l:v-bar}
Let $ \bar v $ be as in the above definition. Then 
\begin{enumerate}[label=(\roman*)]
\item for $ \alpha^i \varphi_i \in \bar V (q, \varepsilon) $ there holds $ \Vert \bar v \Vert \lesssim \frac{1}{\lambda^2} \simeq \tau $ ; 
\item if $ u \in V(q, \varepsilon) $ is such that $ \partial J_\tau(u) = 0 $, 
then $ \alpha^i \varphi_i \in \bar V (q, \varepsilon) $ and $ u=\alpha^{i}\varphi_{i}+\bar v $. 
\end{enumerate}
Moreover for $ \alpha^i \varphi_i \in \bar V (q, \varepsilon) $ we have
\begin{equation}\label{gradient_small_o_1/lamdbda^2}
\partial J_{\tau}(\alpha^{i}\varphi_{i}+\bar v)
=
O(\frac{\tilde \epsilon}{\lambda^{3}}), \quad 
\hbox{ where } \tilde{\epsilon} \longrightarrow 0\hbox{ as } \varepsilon \longrightarrow 0. 
\end{equation}
\end{lemma}
\begin{proof}
We first justify Definition \ref{d:barv}, which amounts to solving in a unique way
\begin{equation}\label{bar_v_to_solve}
\Pi_{H_{\alpha^i \varphi_i}} \partial 
J_\tau(\alpha^i \varphi_i + \bar v) = 0
\end{equation}
denoting by 
$$ 
\Pi_{H_{\alpha^i \varphi_i}} 
\hat{=} 
\quad \text{ the projection onto } \quad  
H_{\alpha^i \varphi_i}.
$$ 
Equivalently, 
as
 $ \partial^2 J_\tau $ 
is invertible on the latter subspace, 
\begin{equation*}
\begin{split}
\bar{v} 
= 
- 
&
(H_{\alpha^i \varphi_i} \partial^2 J_\tau (\alpha^i \varphi_i))^{-1} 
\\ & 
\left[ 
\partial J_\tau(\alpha^i \varphi_i) + 
\left( \partial 
J_\tau(\alpha^i \varphi_i + \bar v) - \partial J_\tau(\alpha^i \varphi_i) 
- \partial^2 J_\tau (\alpha^i \varphi_i) \bar v \right) \right]. 
\end{split}
\end{equation*}
Note, that by Lemma \ref{lem_upper_bound} and for $ \alpha^i \varphi_i \in \bar V (q, \varepsilon) $ 
we have 
$$
\|\partial J_\tau(\alpha^i \varphi_i) \| 
\lesssim 
\frac{1}{\lambda^2}.$$
Moreover 
\begin{equation*}
J_\tau(\alpha^i \varphi_i + \bar v) 
-
\partial J_\tau(\alpha^i \varphi_i) 
-
\partial^2 J_\tau (\alpha^i \varphi_i) \bar v
=
o(\Vert \bar v \Vert)
\end{equation*}
 by H\"older  continuity. Hence
we may use a contraction argument in a ball 
 $$ 
B_{\frac{C}{\lambda^{2}}}(0)\subset H_{\alpha^{i}\varphi_{i}}
 $$ 
to obtain
existence of some $ \bar v $ solving \eqref{bar_v_to_solve} andsatisfyingestimate $ (i) $. Uniqueness follows from the aforementioned invertibility. Hence we have justified Definition \ref{d:barv}. We are left with proving (ii).

\

\noindent
By the definition of $ \bar v $ and the above contraction argument we have that 
\begin{equation}\label{bar_v_solution_to}
\partial^{2}J_{\tau}(\alpha^{i}\varphi_{i})\bar v
=
-\partial J_{\tau}(\alpha^{i}\varphi_{i})+o(\frac{1}{\lambda^2})
\; \text{ on }\; \langle \phi_{k, i}\rangle^{\perp_{L_{g_{0}}}}.
\end{equation}
Testing thus $ \partial J_{\tau}(\alpha^{i}\varphi_{i}) $ on $ \langle \phi_{k, i}\rangle $, we find from Lemmata
\ref{lem_alpha_derivatives_at_infinity}, \ref{lem_lambda_derivatives_at_infinity} and \ref{lem_a_derivatives_at_infinity}
\begin{equation*}
\vert \partial J_{\tau}(\alpha^{i}\varphi_{i})\phi_{k, i}\vert \leq \frac{\tilde \epsilon}{\lambda^{3}}
\quad \text{ for } \quad \alpha^i \varphi_i \in \bar V (q, \varepsilon). 
\end{equation*}
It is easy to see from \eqref{second_variation_evaluating} and Lemma \ref{lem_emergence_of_the_regular_part} that $ \partial^2 J_\tau \phi_{k, i} = o(\frac{1}{\lambda}) $, and 
since $ \| \bar v \| \lesssim \frac{1}{\lambda^2} $ we have that 
\begin{equation}\label{eq:claim-orth}
\partial^{2} J_{\tau}(\alpha^{i}\varphi_{i})\bar v\phi_{k, i}=o(\frac{1}{\lambda^{3}}), 
\end{equation}
More generally we also finds, that for any $ \theta \in (0, 1) $
\begin{equation*}
\partial^{2}J(\alpha^{i}\varphi_{i}+\theta \bar v)\bar v \phi_{k, j}
=
o(\frac{1}{\lambda^{3}}).
\end{equation*}
To see this, since $ \bar v\in \langle \phi_{k, i}\rangle^{\perp_{L_{g_{0}}}} $, recalling \eqref{second_variation_evaluating} 
it is sufficient to show that 
 $$ 
\int K(\alpha^{i}\varphi_{i}+ \theta \bar v)^{p-1} \bar v \varphi_{j}d\mu_{g_{0}} - \int K(\alpha^{i}\varphi_{i})^{p-1} \bar v \varphi_{j} d\mu_{g_{0}} = O(\frac{1}{\lambda^3}). 
 $$ 
This in return can be verified by dividing the domain of integration into $ \{ |\bar v| \leq \alpha^{i}\varphi_{i} \} $ 
and its complementary set, using H\"older  inequality and the fact that 
 $ \Vert \bar v \Vert \lesssim \frac{1}{\lambda^2} $.
Consequently 
\begin{equation*} 
\partial J_{\tau}(\alpha^{i}\varphi_{i}+\bar v)
=
\partial J_{\tau}(\alpha^{i}\varphi_{i}+\bar v)\lfloor_{\langle \phi_{k, i}\rangle }
=
\partial J_{\tau}(\alpha^{i}\varphi_{i})\lfloor_{\langle \phi_{k, i}\rangle }
+
o(\frac{1}{\lambda^{3}})
=
O(\frac{\tilde \epsilon}{\lambda^{3}}), 
\end{equation*}
where $ \tilde{\epsilon} $ tends to zero as $ \varepsilon $ does. 
Finally, if a solution $ \partial J_{\tau}(u)=0 $ exists on $ V(q, \varepsilon) $, then we may write
\begin{equation*}
u=\alpha^{i}\varphi_{i}+\bar v+\tilde v \; \text{ with }\; \tilde v \perp_{L_{g_{0}}}\langle \phi_{k, i}\rangle. 
\end{equation*}
But then
\begin{equation*}
\begin{split}
0
= &
\partial J_{\tau}(\alpha^{i}\varphi_{i}+\bar v+\tilde v)\tilde v
=
\partial J_{\tau}(\alpha^{i}\varphi_{i}+\bar v)\tilde v
+
\partial^{2} J_{\tau}(\alpha^{i}\varphi_{i}+\bar v)\tilde v\tilde v
+
o(\vert\tilde v\vert^{2}), 
\end{split}
\end{equation*}
whence necessarily $ \tilde v=0 $ by uniform positivity  
$$ 
\partial^{2} J_{\tau}(\alpha^{i}\varphi_{i}) 
\quad \text{ on } \quad 
\langle \phi_{k, i}\rangle^{\perp_{L_{g_{0}}}}.
$$
Thus 
\begin{equation*}
\partial J_{\tau}(u)=0 \;\text{ with }\; u \in \bar V(q, \varepsilon) \quad 
\Longrightarrow\quad
u=\alpha^{i}\varphi_{i}+\bar v
\end{equation*}
where
 $ \bar v=\bar v_{\alpha, a, \lambda} $ is
 the unique solution to\eqref{bar_v_minimization_problem}, for which 
 $ \alpha^{i}\varphi_{i}+\bar v \in \bar V(q, \varepsilon) $. 
\end{proof}

\begin{remark}\label{lem_refined_gradient} 
For 
$ \alpha^{i}\varphi_{i}\in \bar V(q, \varepsilon) $ 
and 
$ \Vert \nu \Vert=1 $ 
we have
\begin{equation*}
\begin{split}
\hspace{9pt}& \hspace{-11pt}
\frac{(k_{\tau})_{\alpha^{i}\varphi_{i}}^{\frac{2}{p+1}}}{8n(n-1)}
\partial J_{\tau}(\alpha^{i}\varphi_{i})\nu \\
= & 
-
\alpha^{i} 
\tau 
\underset{B_{\varepsilon}(a_{i})}{\int}
\biggr(
\varphi_{i}^{\frac{n+2}{n-2}}\ln(1+\lambda_{i}^{2}r^{2})^{\frac{n-2}{2}}
-
\frac{\bar c_{1}}{c_{1}}\varphi_{i}^{\frac{n+2}{n-2}} 
+
\frac{2}{n-2}\frac{\tilde c_{1}}{c_{2}}
\varphi_{i}^{\frac{4}{n-2}}\lambda_{i}\partial_{\lambda_{i}}\varphi_{i} 
\biggr)
\nu
d\mu_{g_{0}} 
\\
&
+
\alpha^{i}
\tau 
\underset{B_{\varepsilon}(a_{i})}{\int}
\biggr(
\frac{\tilde c_{1}}{\tilde c_{2}}\frac{ \lambda_{i}^{2}r^{2}}{2n}\varphi_{i}^{\frac{n+2}{n-2}}
 -
\frac{\tilde c_{1}\bar c_{2}}{\tilde c_{2}c_{1}}\varphi_{i}^{\frac{n+2}{n-2}} 
+
\frac{2}{n-2}\frac{\tilde c_{1}}{c_{2}}\varphi_{i}^{\frac{4}{n-2}}\lambda_{i}\partial_{\lambda_{i}}\varphi_{i}
\biggr)\nu d\mu_{g_{0}}
\\
& -
\alpha^{i}
\int_{B_{\varepsilon}(a_{i})}
(
\frac{\nabla^{2}_{k, l}K_{i}}{2K_{i}}
x^{k}x^{l}
-
\frac{\Delta K_{i}}{2nK_{i}}r^{2}
)
\varphi_{i}^{\frac{n+2}{n-2}}\nu d\mu_{g_{0}} + o(\frac{1}{\lambda^{2}}), 
\end{split} 
\end{equation*}
referring to the table at the end of the paper for the definition of the constants. 
As a consequence of these formulae
one can prove that $ \bar v $ is indeed of order $ \frac{1}{\lambda^2} $ and not smaller, 
as well as determine the leading order in its expansion. In any case due to some cancellation 
properties this will not substantially affect the eigenvalues of the Hessian of $ J_\tau $ 
at $ \alpha^{i}\varphi_{i} + \bar v $, which we estimate in the next section. 
\end{remark}

\medskip

Let us set
 $
(d_{1, i}, d_{2, i}, d_{3, i})=(1, -\lambda_{i}\partial_{\lambda_{i}}, \frac{1}{\lambda_{i}}\nabla_{a_{i}})
 $
for $ i= 1, \ldots, q $. 

\begin{lemma}
\label{lem_bar_v_estimates}
For $ u=\alpha^{i}\varphi_{i}+\bar v\in \bar V(q, \varepsilon) $ there holds
\begin{equation*}
\Vert \bar v \Vert, \Vert d_{l, j}\bar v\Vert=O(\frac{1}{\lambda^{2}}).
\end{equation*}
\end{lemma}
\begin{proof}
The bound on $ \Vert \bar v \Vert $ follows from Lemma \ref{lem_refined_gradient}. 
Differentiating 
$$ \langle \phi_{k, i}, \bar v \rangle_{L_{g_{0}}} = 0 $$ 
we obtain 
\begin{equation*}
\langle \phi_{k, i}, d_{l, j}\bar v \rangle_{L_{g_{0}}}
=
-\langle d_{l, j}\phi_{k, i}, \bar v \rangle_{L_{g_{0}}}
=
O(\Vert \bar v \Vert), 
\end{equation*}
whence denoting by $ \Pi_{\langle \phi_{k, i}\rangle} $ 
the orthogonal projection onto $ \Pi_{\langle \phi_{k, i}\rangle} $ we have 
 $$ 
\Vert \Pi_{\langle \phi_{k, i}\rangle}\bar v \Vert\simeq \frac{1}{\lambda^{2}}
\quad \text{ due to } \quad 
\Vert \bar v \Vert \lesssim \frac{1}{\lambda^{2}}. $$ 
Moreover, since 
 $ 
\partial J_{\tau}(\alpha^{i}\varphi_{i}+\bar v)v=0
 $ 
for every smoothly varying vector field $ v\in \langle \phi_{k, i}\rangle^{\perp_{L_{g_{0}}}} $ 
of unit norm we have 
\begin{equation*}
\begin{split}
0
= &
d_{l, j}
(\partial J_{\tau}(\alpha^{i}\varphi_{i}+\bar v)v)
=
\partial^{2} J_{\tau}(\alpha^{i}\varphi_{i}+\bar v)d_{l, j}(\alpha^{i}\varphi_{i}+\bar v)v
+
\partial J_{\tau}(\alpha^{i}\varphi_{i}+\bar v)d_{l, j}v
\end{split}
\end{equation*}
and we can estimate the last summand above as 
\begin{equation*}
\begin{split}
\partial J_{\tau}(\alpha^{i}\varphi_{i}+\bar v)d_{l, j}v
= &
\partial J_{\tau}(\alpha^{i}\varphi_{i}+\bar v)\Pi_{\langle \phi_{k, i}\rangle }(d_{l, j}v) 
= 
O(\vert \partial J_{\tau}(\alpha^{i}\varphi_{i}+\bar v)\vert\Vert v \Vert), 
\end{split}
\end{equation*}
since 
$$\langle \phi_{k, i}, d_{l, j}v \rangle=\langle d_{l, j}\phi_{k, i}, v\rangle=O(\Vert v \Vert).$$ 
Thence $ \partial J_{\tau}(\alpha^{i}\varphi_{i}+\bar v)=O(\frac{1}{\lambda^{2}}) $ implies
\begin{equation*}
\partial^{2}J_{\tau}(\alpha^{i}\varphi_{i}+\bar v) v d_{l, j}\bar v
=
-
\partial^{2}J_{\tau}(\alpha^{i}\varphi_{i}+\bar v) v d_{l, j}(\alpha^{i}\varphi_{i}) 
+
O(\frac{1}{\lambda^{2}}).
\end{equation*}
Then the claim would follow from
 \begin{equation*}
\Vert \Pi_{\langle \phi_{k, i}\rangle} (d_{l, j} \bar v) \Vert\simeq \frac{1}{\lambda^{2}}, 
\end{equation*}
which we had seen before, and the 
uniform positivity  
$$ \partial^{2}J_{\tau}(\alpha^{i}\varphi_{i})>0
\quad \text{ on } \quad \langle \phi_{k, i} \rangle^{\perp_{L_{g_{0}}}},$$ 
provided we show 
\begin{equation}\label{second_variation_nu_phi_k_i_interaction_improved}
\begin{split}
\partial^{2}J_{\tau}(\alpha^{i}\varphi_{i}+\bar v)\phi_{l, j} v
= &
O(\frac{1}{\lambda^{2}}), 
\end{split}
\end{equation}
cf. 
\eqref{second_variation_nu_phi_1_interaction},
\eqref{second_variation_nu_phi_2_phi_3_interaction} 
for weaker statements. Let us prove \eqref{second_variation_nu_phi_k_i_interaction_improved}
for $ l=1 $. We claim
\begin{equation*}
\partial^{2}J_{\tau}(\alpha^{i}\varphi_{i}+\bar v)\varphi_{j} v
= 
\partial^{2}J_{\tau}(\alpha^{i}\varphi_{i})\varphi_{j} v + O(\frac{1}{\lambda^2}). 
\end{equation*}
From \eqref{second_variation_evaluating}, since $ v\in \langle \phi_{k, i}\rangle^{\perp_{L_{g_{0}}}} $, 
it is sufficient to show that we must show, cf. the proof of Lemma \ref{l:v-bar}, 
 $$ 
\int K(\alpha^{i}\varphi_{i}+\bar v)^{p-1}v \varphi_{j}d\mu_{g_{0}} - \int K(\alpha^{i}\varphi_{i})^{p-1} v \varphi_{j} d\mu_{g_{0}} = O(\frac{1}{\lambda^2}). 
 $$ 
Again this can be seen considering the set $ \{ |\bar v| \leq \alpha^{i}\varphi_{i} \} $ 
and its complementary, using H\"older  inequality and 
 $ \Vert \bar v \Vert \lesssim \frac{1}{\lambda^2} $. 
Thus from the above claim and\eqref{second_variation_evaluating} we find, due to the orthogonalities $ \langle \phi_{k, i}, v\rangle_{L_{g_0}}=0 $, 
\begin{equation*}
\begin{split}
\partial^{2} J_{\tau}(\alpha^{i}\varphi_{i})\varphi_{i}v
= &
-\frac{2p}{(k_{\tau})_{\alpha^{i}\varphi_{i}}^{\frac{2}{p+1}}}
\frac{r_{\alpha^{i}\varphi_{i}}}{(k_{\tau})_{\alpha^{i}\varphi_{i}}}\int K(\alpha^{i}\varphi_{i})^{p-1}\varphi_{j}vd\mu_{g_{0}} \\
 & -
\frac{4}{(k_{\tau})_{\alpha^{i}\varphi_{i}}^{\frac{2}{p+1}+1}}
\int L_{g_{0}}(\alpha^{i}\varphi_{i})\varphi_{j}d\mu_{g_{0}}\int K(\alpha^{i}\varphi_{i})^{p}vd\mu_{g_{0}} \\
& +
\frac{2(p+3)r_{\alpha^{i}\varphi_{i}}}{(k_{\tau})_{\alpha^{i}\varphi_{i}}^{\frac{2}{p+1}+2}}
\int K(\alpha^{i}\varphi_{i})^{p}\varphi_{j}d\mu_{g_{0}}\int K(\alpha^{i}\varphi_{i})^{p}vd\mu_{g_{0}}.
\end{split}
\end{equation*}
By definition of $ \bar V(q, \varepsilon) $ we have $ \tau\simeq\frac{1}{\lambda^{2}} $ and recalling 
\eqref{intergral_sum_of_bubble_nonlinear_evaluated}
and
\eqref{r_alpha_delta_expansion}
we may simplify this to
\begin{equation*}
\begin{split}
\partial^{2} & J_{\tau}(\alpha^{i}\varphi_{i})\varphi_{j}v
\simeq 
-4n(n-1)\frac{n+2}{n-2}
\frac{\alpha^{2}}{\alpha_{K, \tau}^{\frac{2n}{n-2}}}\int K(\alpha^{i}\varphi_{i})^{\frac{4}{n-2}}\varphi_{j}vd\mu_{g_{0}} \\
& -
\frac{2}{\bar c_{0}\alpha_{K, \tau}^{\frac{2n}{n-2}}}
\int L_{g_{0}}(\alpha^{i}\varphi_{i})\varphi_{j}d\mu_{g_{0}}\int K(\alpha^{i}\varphi_{i})^{\frac{n+2}{n-2}}vd\mu_{g_{0}} \\
& +
4n(n-1)\frac{(\frac{n+2}{n-2}+3)\alpha^{2}}{\bar c_{0}(\alpha_{K, \tau}^{\frac{2n}{n-2}})^{2}}
\int K(\alpha^{i}\varphi_{i})^{\frac{n+2}{n-2}}\varphi_{j}d\mu_{g_{0}}
\int K(\alpha^{i}\varphi_{i})^{\frac{n+2}{n-2}}vd\mu_{g_{0}}
\end{split}
\end{equation*}
up to error $ O(\frac{1}{\lambda^{2}}) $. Moreover from
\eqref{L_g_0_bubble_interaction} and 
\eqref{single_bubble_L_g_0_integral_expansion_exact} we have
\begin{equation*}
\int L_{g_{0}}(\alpha^{i}\varphi_{i})\varphi_{j}d\mu_{g_{0}}
=
4n(n-1)\bar c_{0}\alpha_{j}+O(\frac{1}{\lambda^{2}})
\end{equation*}
and since $ d(a_{i}, a_{j}) \simeq 1 $, we find by expanding and using Lemma \ref{lem_interactions}
\begin{enumerate}[label=(\roman*)]
 \item \quad 
 $ 
 \int K(\alpha^{i}\varphi_{i})^{\frac{4}{n-2}}\varphi_{j}vd\mu_{g_{0}}
=
\alpha_{j}^{\frac{4}{n-2}}\int K\varphi_{j}^{\frac{n+2}{n-2}}vd\mu_{g_{0}}; \quad 
 $ 
 \item \quad
 $ 
 \int K(\alpha^{i}\varphi_{i})^{\frac{n+2}{n-2}}vd\mu_{g_{0}} 
=
\sum_{i}\alpha_{i}^{\frac{n+2}{n-2}}
\int K\varphi_{i}^{\frac{n+2}{n-2}}vd\mu_{g_{0}}; 
 $ 
 \item \quad
 $ \int K(\alpha^{i}\varphi_{i})^{\frac{n+2}{n-2}}\varphi_{j}d\mu_{g_{0}} 
=
\alpha_{j}^{\frac{n+2}{n-2}}\int K\varphi_{j}^{\frac{2n}{n-2}}d\mu_{g_{0}}; \quad 
 $ 
 \item \quad
 $ 
\int K(\alpha^{i}\varphi_{i})^{\frac{n+2}{n-2}}vd\mu_{g_{0}} 
=
\sum_{i}\alpha_{i}^{\frac{n+2}{n-2}}\int K\varphi_{i}^{\frac{n+2}{n-2}}vd\mu_{g_{0}}, 
 $ 
\end{enumerate}
up to some $ O(\frac{1}{\lambda^{2}}) $. Therefore, since
 $ \frac{\vert \nabla K_{i}\vert}{\lambda_{i}}=O(\frac{1}{\lambda^{2}}) $ due to\eqref{def_refined_neighbourhood}, we obtain
\begin{equation*}
\begin{split}
\partial^{2} J_{\tau}(\alpha^{i}\varphi_{i})\varphi_{j}v
\simeq &
-4n(n-1)\frac{n+2}{n-2}
\frac{\alpha^{2}}{\alpha_{K, \tau}^{\frac{2n}{n-2}}}K_{i}\alpha_{j}^{\frac{4}{n-2}}\int \varphi_{j}^{\frac{n+2}{n-2}}vd\mu_{g_{0}} \\
& -
\frac{8n(n-1)\alpha_{j}}{\alpha_{K, \tau}^{\frac{2n}{n-2}}}
\sum_{i}K_{i}\alpha_{i}^{\frac{n+2}{n-2}}
\int \varphi_{i}^{\frac{n+2}{n-2}}vd\mu_{g_{0}} \\
& +
4n(n-1)\frac{(\frac{n+2}{n-2}+3)\alpha^{2}}
{(\alpha_{K, \tau}^{\frac{2n}{n-2}})^{2}}
\alpha_{j}^{\frac{n+2}{n-2}}K_{j}
\sum_{i}K_{i}\alpha_{i}^{\frac{n+2}{n-2}}
\int \varphi_{i}^{\frac{n+2}{n-2}}vd\mu_{g_{0}}
\end{split}
\end{equation*}
up to an error $ O(\frac{1}{\lambda^{2}}) $. Therefore using again \eqref{def_refined_neighbourhood}
 we have 
\begin{equation*}
\begin{split}
\partial^{2}J_{\tau}(\alpha^{i}\varphi_{i})\varphi_{j}v
\simeq &
-\frac{n+2}{n-2}
\int \varphi_{j}^{\frac{n+2}{n-2}}vd\mu_{g_{0}} 
-
2
\sum_{i}\frac{\alpha_{i}\alpha_{j}}{\alpha^{2}}
\int \varphi_{i}^{\frac{n+2}{n-2}}vd\mu_{g_{0}} \\
& +
(\frac{n+2}{n-2}+3)
\sum_{i}\frac{\alpha_{i}\alpha_{j}}{\alpha^{2}}
\int \varphi_{i}^{\frac{n+2}{n-2}}vd\mu_{g_{0}}
\end{split}
\end{equation*}
up to the same error.
Thus 
$$ 
\partial^{2}J_{\tau}(\alpha^{i}\varphi_{i})\varphi_{j}v=O(\frac{1}{\lambda^{2}}) 
$$ 
using \eqref{non_linear_v_part_interaction}, 
obtaining \eqref{second_variation_nu_phi_k_i_interaction_improved} for $ l = 1 $. 
For $ l=2, 3 $ the reasoning is analogous.
\end{proof}

\medskip
Theorem \ref{t:ex-multi}follows from the next proposition, based on the analysis 
of Section \ref{s:2nd}, and Corollary \ref{cor_restricted_second_variation}.

\begin{proposition}\label{prop_subcritical_existence} 
Let $ n\geq 5 $ and let $ K : M \longrightarrow \R $ be a positive Morse function satisfying \eqref{eq:nd}.
Then for every subset 
 $$ \{x_{1}, \ldots, x_{q}\} \subseteq\{\nabla K=0\}\cap \{\Delta K<0\} $$ 
and, as $ \tau \longrightarrow 0 $, there exists a unique
$u=\alpha^{i}\varphi_{a_{i}, \lambda_{i}}+\bar v \in V(q, \varepsilon)$ with
\begin{equation*}
\Vert u \Vert^{2}_{L_{g_{0}}}=1, \; d(a_{i}, x_{i})=o(1)\; \text{ and }\; \partial J_\tau(u)=0.
\end{equation*}
\end{proposition}

\begin{proof}
Due to \eqref{gradient_small_o_1/lamdbda^2} we have 
\begin{equation*}
\vert \partial J\vert \leq \frac{\tilde \epsilon}{\lambda^{3}}\; \text{ on }\; \bar V(q, \varepsilon)
\; \text{ and }\; 
\vert \partial J\vert \geq \frac{\hat \epsilon}{\lambda^{3}}
\; \text{ on }\; \partial \bar V(q, \varepsilon)
\end{equation*}
as long as $ c<\alpha_{j}<C $. Thus by $ (ii) $ in Lemma \ref{l:v-bar} it is sufficient to 
look for critical points in the set 
\begin{equation*}
\tilde{\mathcal{C}} =\{ \tilde u (\alpha, \lambda, a) = \alpha^{i}\varphi_{i}+\bar v (\alpha, \lambda, a)\in \bar V (q, \varepsilon) \mid \Vert \tilde u \Vert^{2}_{L_{g_{0}}}=1 \}, 
\end{equation*}
which is a smooth $ (3(n+2)-1) $-dimensional manifold in $ W^{1, 2}(M, g_0) $. 

Vice-versa we claim that a critical point of $ J_\tau \lfloor_{\tilde{\mathcal{C}}} $ 
is indeed a critical point of $ J_\tau $. In fact by Lagrange multiplier rule
the gradient of $ J_\tau $ at a constrained critical point $ \tilde{u}_0 $ must be orthogonal 
to $ \tilde{\mathcal{C}} $. Since $ J_\tau $ is scaling invariant, its 
gradient on $ \mathcal{C} $ must be tangent to the unit sphere in the 
 $ \| \cdot \|_{L_{g_0}} $ norm. On the other hand, by construction of 
 $ \bar v $, the gradient of $ J_\tau $ at $ \tilde{u}_0 $ is tangent to 
 $$ 
\mathcal{C} =\{\alpha^{i}\varphi_{i} \in \bar V (q, \varepsilon) \mid \Vert u \Vert^{2}_{L_{g_{0}}}=1 \}
 $$ 
at the point $ u_0 $ such that $ \tilde{u}_0 = u_0 + \bar v_0 $.
By the estimate on the derivatives of $ \bar v $ in Lemma \ref{lem_bar_v_estimates}, 
 $ T_{\tilde{u}_0} \tilde{\mathcal C} $ is nearly parallel to $ T_{{u}_0} {\mathcal C} $, 
which implies that $ \partial J_\tau(\tilde{u}_0) = 0 $, as desired. 

It remains to prove existence and uniqueness of critical points of $ J_\tau \lfloor_{\tilde{\mathcal{C}}} $. 
For the existence part we may use the expansions in Lemmas \ref{lem_alpha_derivatives_at_infinity}, 
\ref{lem_lambda_derivatives_at_infinity} and \ref{lem_a_derivatives_at_infinity} together 
with the definition of $ \bar V(q, \varepsilon) $ to show, that $ \partial J_\tau $ is non-vanishing on the 
boundary of $ \tilde{\mathcal C} $. For example, cf. (iii) in the definition of $ \bar V(q, \varepsilon) $, 
suppose
 $$ 
 \lambda_{j}^{2}
 = - 
 c_{2}\frac{\Delta K(x_{j})}{K(x_{j})\tau}
+ \frac{\varepsilon}{\lambda^{2}}; \quad \quad \frac{1}{\lambda^2} = \tau. 
 $$ 
From Lemma \ref{lem_lambda_derivatives_at_infinity} we deduce, that there exists 
 $ \tilde{\epsilon} > 0 $, tending to zero as $ \varepsilon \longrightarrow 0 $, such that 
 $$ 
\lambda_{j} \partial_{\lambda_j} J_\tau (\alpha^i \varphi_i) > \frac{\tilde{\epsilon}}{\lambda^3}. 
 $$ 
From Lemmas \ref{l:v-bar} and \ref{lem_bar_v_estimates} we also have
 $$ 
\lambda_{j} \partial_{\lambda_j} J_\tau (u (\alpha, \lambda, a)) 
>
\frac{1}{2}\frac{\tilde{\epsilon}}{\lambda^3}
 $$ 
and a similar reversed inequality with opposite sign, if 
 $$ \lambda_{j}^{2}
= 
- 
c_{2}\frac{\Delta K(x_{j})}{K(x_{j})\tau}
- 
\frac{\varepsilon}{\lambda^{2}}.
 $$ 
Analogous estimates are derived for the $ \alpha- $ and $ a- $ derivatives, yielding that the degree of 
 $ \partial J_\tau $ on $ \tilde{\mathcal C} $ is well-defined and non-zero. This shows the existence of 
a critical point for $ J_\tau \lfloor_{\tilde{\mathcal{C}}} $, which is (freely) critical for $ J_\tau $ by the 
above discussion. Since by construction the negative part of the above solutions 
is small in $ W^{1, 2} $ norm, it is possible to show from Sobolev  inequality that it has to 
vanish identically, so full positivity follows then from the maximum principle. 

Uniqueness follows from Lemma \ref{lem_bar_v_estimates} and Proposition \ref{lem_refined_second_variation}, implying the strict convexity or concavity of $ J_\tau \lfloor_{\tilde{\mathcal{C}}} $ with respect to all parameters $ \alpha $  , $ \lambda $   and the 
coordinates of the points $ a_i $, provided they are chosen so that $ \nabla^2 K(x_i) $ is diagonal. 
\end{proof}

\section{The second variation} \label{s:2nd}

Let $ \bar V(q, \varepsilon) $ be the open set defined in \eqref{def_refined_neighbourhood}.
The aim of this section is to find there a nearly diagonal form of the second 
differential of $ J_\tau $. We recall our notation from Section \ref{s:prel}, in 
particular that of the orthogonal space $ H_u $ in \eqref{eq:Hu}.

\begin{proposition}\label{lem_refined_second_variation}
For $ \alpha^{i}\varphi_{i}+\bar v\in \bar V(q, \varepsilon) $, 
consider the decomposition 
\begin{equation*}
\begin{split}
W^{1, 2}(M, g_0)
= &
H_{\alpha^i \varphi_{i}} 
\oplus
\langle \varphi_{i} \rangle_{1\leq i\leq q}
\oplus
\langle \lambda_{i}\partial_{\lambda_{i}}\varphi_{i}\rangle_{1\leq i\leq q}
\oplus
\langle \frac{\nabla_{a_{i}}}{\lambda_{i}}\varphi_{i}\rangle_{1\leq i\leq q}
\\
= &: \mathcal{V} \oplus X_\alpha \oplus X_\lambda \oplus X_a. 
\end{split}
\end{equation*} 
Then there exists a basis $ \mathbb{B} $ of $ W^{1, 2}(M, g_0) $ with elements in the subspaces 
of the above decomposition, such that the 
coefficients of the 
the second differential of $ J_\tau $ with respect to $ \mathbb{B} $ have the {form} 
\begin{equation*} 
[\partial^{2}J_{\tau}(\alpha^{k}\varphi_{k}+\bar v)]_{\mathbb{B}} 
=\frac{1}{\lambda^2}
 \begin{pmatrix}
\mathbb{V}_{+} & 0 & 0 & 0 \\
0 &\mathbb{A}_{q-1, 0} & 0 & 0\\
0 &0 & \mathbb{\Lambda}_{+}\\
0 & 0 & 0 &
-\frac{\nabla^{2}K}{K}\\
\end{pmatrix}
+
o(\frac{1}{\lambda^2}), \quad \text{ where } 
\end{equation*}
\begin{enumerate}[label=(\roman*)]
 \item $ \mathbb{\Lambda}_{+} $ represents the coefficients of a symmetric, positive-definite 
 operator on $ \mathcal{V} $ with eigenvalues uniformly bounded away from zero; 
 \item $ \mathbb{A}_{q-1, 0} $ has $ q-1 $ negative eigenvalues uniformly bounded away from zero and one-dimensional kernel; 
 \item $ \mathbb{\Lambda}_{+} $ is positive-definite with eigenvalues uniformly bounded away from zero; 
 \item $ -\frac{\nabla^{2}K}{K} $ stands for the diagonal matrix
 $ -(\frac{\nabla^{2} K_{i}}{K_{i}})_{i=1, \ldots, q} $. 
\end{enumerate}
\end{proposition}

\medskip

\begin{remark}\label{r:basis}
The basis elements in $ \mathbb{B} $ corresponding to the first two blocks have norms of order $ \frac{1}{\lambda^2} $, 
while the ones corresponding to the last two blocks have norm of order $ 1 $. We made this choice to guarantee 
the {\em off-diagonal} terms in the above matrix to be of order $ o(\frac{1}{\lambda^2}) $.
\end{remark}

\medskip

\begin{proof} 
We will analyse
\eqref{second_variation_evaluating} 
for $ u=\alpha^{i}\varphi_{i}+\bar v \in \bar V(q, \varepsilon) $. Recall from Section \ref{s:prel}
\begin{equation*}
W^{1, 2}(M, g_{0})
=
\langle \phi_{k, i} \rangle_{k, i} \oplus H_{\alpha^{i} \varphi_{i}}, 
\end{equation*}
We then choose a
 $ \langle \cdot, \cdot \rangle^{L_{g_{0}}} $-orthonormal 
basis $ \{ \nu_{0}, \nu_{1}, \nu_{2}, \ldots \} $ for $ H_{\alpha^i \varphi_{i}} $ 
and for some $ \lambda\simeq \lambda_{i}\simeq \frac{1}{\sqrt{\tau}} $ 
define 
\begin{equation*}
\mathbb{B}=
\{
\tilde \phi_{k, i}, \tilde \nu_{j}
\}
=
\{\frac{\varphi_{i}}{\lambda_{i}}, \lambda_{i}\partial_{\lambda_{i}}\varphi_{i}, \frac{\nabla_{a_{i}}}{\lambda_{i}}\varphi_{i}, \frac{\nu_{j}}{\lambda}\}; 
\quad \quad k = 1, 2, 3, \quad i = 1, \dots, q. 
\end{equation*} 
 With this choiceit is not hard to see that the coefficients $ [\partial^{2}J_{\tau}(\alpha^{k}\varphi_{k}+\bar v)]_{\mathbb{B}} $ are all of order 
 $ O(\frac{1}{\lambda^2}) $, and our goal is to make their estimates more precise, considering different matrix blocks.

\

\noindent {\bf{First block.}}
The fact that $ \partial^2 J_\tau(\alpha^i \varphi_{i}) $ is (uniformly) positive-definite on $ H_{\alpha^i \varphi_{i}} $ is well-known, 
see e.g. \cite{bab}. The positivity of $ \partial^2 J_\tau(\alpha^i \varphi_{i} + \varepsilon_{v}) $ on the same subspace 
follows from the H\"older continuity of the second differential and the fact that $ \Vert \bar v \Vert=O(\frac{1}{\lambda^{2}}) $. 

\noindent {\bf{First two blocks.}}
Testing the second differential with $ \tilde{\nu}_{i} $ and $ \tilde\phi_{1, j}=\frac{\varphi_{j}}{\lambda} $ we get
\begin{equation}\label{second_variation_nu_phi_1_interaction}
\begin{split}
\partial^{2} J_{\tau}(\alpha^{i}\varphi_{i}+\bar v) \tilde \nu_{i}\tilde \phi_{1, j}
=o(\frac{1}{\lambda^{2}})
\end{split}
\end{equation}
using 
Lemma \ref{lem_emergence_of_the_regular_part}
, 
$ \| \bar v \| \lesssim \frac{1}{\lambda^2} $ 
and the orthogonality 
$$ \langle \tilde \nu_{i}, \tilde \phi_{1, j}\rangle_{L_{g_{0}}}=0. $$
Moreover from \eqref{second_variation_evaluating} and 
$\Vert \tilde\phi_{1, i}\Vert=O(\frac{1}{\lambda})$ we find with
$c_{0}=\int_{\R^{n}}\frac{dx}{(1+r^{2})^{n}}$ 
\begin{equation}\label{second_derivative_alpha_space}
\begin{split}
\partial^{2} J_{\tau}(\alpha^{k}\varphi_{k}+\bar v) \tilde\phi_{1, i}\tilde\phi_{1, j}
= &
\frac{16n(n-1)\bar c_{0}^{\frac{2}{n}}}{(n-2)(\alpha_{K, \tau}^{\frac{2n}{n-2}})^{\frac{n-2}{n}}\lambda^{2}}
(
-
\delta_{k, l}
+
\frac{\alpha_{k}\alpha_{l}}{\alpha^{2}}
)
=\mathbb{A}_{i, j} 
\end{split}
\end{equation}
up to an error of order $ o(\frac{1}{\lambda^{2}}) $. Let us compare the 
above expression to
\begin{equation*}
f(\alpha)=\frac{\alpha^{2}}{(\alpha_{K}^{\frac{2n}{n-2}})^{\frac{n-2}{n}}}; 
\quad \quad \alpha := \sum_{i=1}^q \alpha_{i}^2, \quad \alpha_{K}^{\frac{2n}{n-2}} := 
 \sum_{i=1}^q K_i \alpha_{i}^{\frac{2n}{n-2}}
\end{equation*} 
with first- and second-order derivatives given by 
\begin{equation*}
\frac{1}{2}\partial_{\alpha_{i}}f(\alpha)
=
\frac{\alpha_{i}}{(\alpha_{K}^{\frac{2n}{n-2}})^{\frac{n-2}{n}}}
-
\frac{\alpha^{2}K_{i}\alpha_{i}^{\frac{n+2}{n-2}}}{(\alpha_{K}^{\frac{2n}{n-2}})^{\frac{n-2}{n}+1}}
=
\frac{\alpha_{i}}{(\alpha_{K}^{\frac{2n}{n-2}})^{\frac{n-2}{n}}}
(1-\frac{\alpha^{2}}{\alpha_{K}^{\frac{2n}{n-2}}}K_{i}\alpha_{i}^{\frac{4}{n-2}}) 
\end{equation*}
and
\begin{equation*}
\begin{split}
\frac{1}{2}\partial_{\alpha_{i}}\partial_{\alpha_{j}}f(\alpha)
= & 
\delta_{i, j}
\frac{1}{(\alpha_{K}^{\frac{2n}{n-2}})^{\frac{n-2}{n}}}
(1-\frac{n+2}{n-2}\frac{\alpha^{2}}{\alpha_{K}^{\frac{2n}{n-2}}}K_{i}\alpha_{i}^{\frac{4}{n-2}})
\\ &
+
2
\frac{\alpha_{i}\alpha_{j}}{(\alpha_{K}^{\frac{2n}{n-2}})^{\frac{n-2}{n}+1}}
\frac{\alpha^{2}}{\alpha_{K}^{\frac{2n}{n-2}}}K_{i}\alpha_{i}^{\frac{4}{n-2}}K_{j}\alpha_{j}^{\frac{4}{n-2}}
\\
&
-
2
\frac{\alpha_{i}\alpha_{j}}{(\alpha_{K}^{\frac{2n}{n-2}})^{\frac{n-2}{n}+1}}
(K_{i}\alpha_{i}^{\frac{4}{n-2}}+K_{j}\alpha_{j}^{\frac{4}{n-2}})
\\ &
+\frac{2n}{n-2}
\frac{\alpha^{2}}{(\alpha_{K}^{\frac{2n}{n-2}})^{\frac{n-2}{n}+2}}
K_{j}\alpha_{j}^{\frac{n+2}{n-2}}K_{i}\alpha_{i}^{\frac{n+2}{n-2}}.
\end{split}
\end{equation*}
The function $ f $ is scaling invariant andrestricted to 
\begin{equation*}
\{\alpha_{K}^{\frac{2n}{n-2}}=1\}
\end{equation*}
attains its maximumat $ (\alpha_{i})_i $ satisfying 
\begin{equation*}
\frac{\alpha^{2}}{\alpha_{K}^{\frac{2n}{n-2}}}K_{i}\alpha_{i}^{\frac{4}{n-2}}=1
\quad \; \text{ for all }\; i=1, \ldots, q, 
\end{equation*}
where we have
\begin{equation}\label{alpha_functional_second_derivative}
\begin{split}
\frac{1}{2}\partial_{\alpha_{i}}\partial_{\alpha_{j}}f(\alpha)
= &
\frac{4}{(n-2)(\alpha_{K}^{\frac{2n}{n-2}})^{\frac{n-2}{n}}}
(
-\delta_{i, j}
+
\frac{\alpha_{i}\alpha_{j}}{\alpha^{2}}
)
.
\end{split}
\end{equation}
Comparing \eqref{second_derivative_alpha_space} and \eqref{alpha_functional_second_derivative} 
we conclude with obvious notation
\begin{equation*} 
[\partial^{2}J_{\tau}(\alpha^{k}\varphi_{k}+\bar v)]_{\mathbb{B}} 
=
 \begin{pmatrix}
\frac{1}{\lambda^2} \mathbb{V}_{+} & 0 & \partial^{2}J_{\tau}\tilde \nu\tilde \phi_{2} & \partial^{2}J_{\tau} \tilde \nu \tilde \phi_{3} \\
0 & \frac{1}{\lambda^2} \mathbb{A}_{q-1, 0} & \partial^{2}J_{\tau} \tilde \phi_{1} \tilde \phi_{2} &\partial^{2}J_{\tau} \tilde \phi_{1}\tilde \phi_{3} \\
\partial^{2}J_{\tau} \tilde \phi_{2} \tilde \nu & \partial^{2}J_{\tau} \tilde \phi_{2} \tilde \phi_{1} & \partial^{2}J_{\tau} \tilde \phi_{2} \tilde \phi_{2} & \partial^{2}J_{\tau} \tilde \phi_{2} \tilde \phi_{3} \\
\partial^{2}J_{\tau} \tilde \phi_{3}\tilde \nu_{}& \partial^{2}J_{\tau} \tilde \phi_{3}\tilde \phi_{1}& \partial^{2}J_{\tau} \tilde \phi_{3}, \tilde \phi_{2} & \partial^{2}J_{\tau} \tilde \phi_{3}\tilde \phi_{3} \\
\end{pmatrix}
\end{equation*}
up to some $o(\frac{1}{\lambda^{2}})$. 
\

\noindent {\bf{Terms off 2x2 blocks.}} 
Let us consider next the interaction of $ \tilde \nu_{i} $ with 
$$ \tilde \phi_{k, j}=\phi_{k, j} \quad \text{ for } \quad  k=2, 3. $$
Since 
\begin{equation*}
\bar v=O(\frac{1}{\lambda^{2}})
, \;
\tilde \nu_{i}=O(\frac{1}{\lambda})
, \; 
\langle \varphi_{k}, \phi_{k, j}\rangle_{L_{g_{0}}}=O(\frac{1}{\lambda^{2}}) \quad 
\;\text{ and }\; \quad 
\langle \nu_{i}, \phi_{k, j}\rangle_{L_{g_{0}}}=0,
\end{equation*}
we simply find for \eqref{second_variation_evaluating} 
\begin{equation}\begin{split}\label{second_variation_nu_phi_interaction} 
\partial^{2}J_{\tau}(\alpha^{l}\varphi_{l}+\bar v) \tilde \nu_{i}\tilde \phi_{j, k} 
= &
\partial^{2}J_{\tau}(\alpha^{l}\varphi_{l})\tilde\nu_{i}\tilde\phi_{j, k} \\
= &
-\frac{2pr_{\alpha^{i}\varphi_{i}}}{k_{\tau}^{\frac{2}{p+1}+1}}
\int K(\alpha^{l}\varphi_{l})^{p-1}\tilde\nu_{i}\tilde\phi_{j, k}d\mu_{g_{0}}
\end{split}
\end{equation}
up to some $ o(\frac{1}{\lambda^{2}}) $. 
Indeed by \eqref{second_variation_evaluating} the crucial estimates to verify \eqref{second_variation_nu_phi_interaction} are 
\begin{equation}\label{eq:crucial}
\int K(\alpha^{l}\varphi_{l})^{p}\tilde\nu_{i}d\mu_{g_{0}}
=
o(\frac{1}{\lambda^2})
=
\int K(\alpha^{l}\varphi_{l})^{p}\tilde \phi_{k, j}d\mu_{g_{0}}.
\end{equation} 
These however follow easily by expansion and interaction estimates using 
\begin{equation*}
\langle \varphi_{l}, \phi_{k, j}\rangle_{L_{g_{0}}}=O(\frac{1}{\lambda^{2}}), \; \quad
\langle \nu_{i}, \phi_{k, j}\rangle_{L_{g_{0}}}=0, 
\quad 
L_{g_{0}}\varphi_{l}=4n(n-1)\varphi_{l}^{\frac{n+2}{n-2}}+o(1) 
\end{equation*}
in $W^{-1, 2}$ and Lemma \ref{lem_testing_with_v}. For the remaining integral in \eqref{second_variation_nu_phi_interaction} we then have
\begin{equation}\begin{split}\label{example_expansion_interaction}
\int K(\alpha^{l} &\varphi_{l})^{p-1}\tilde\nu_{i}\tilde\phi_{j, k}d\mu_{g_{0}}
= 
K_{j}\int (\alpha^{l}\varphi_{l})^{p-1}\tilde\nu_{i}\tilde\phi_{j, k}d\mu_{g_{0}} +o(\frac{1}{\lambda^{2}}) \\
= &
K_{j}\int_{\{\varphi_{j}>\sum_{j\neq l}\alpha^{l}\varphi_{l}\}} (\alpha^{l}\varphi_{l})^{p-1}\tilde\nu_{i}\tilde\phi_{j, k}d\mu_{g_{0}} \\
& +
O
(
\frac{1}{\lambda}
\sum_{j\neq l}
\Vert \varphi_{l}^{p-1}\varphi_{j}\Vert_{L^{\frac{p+1}{p}}}
)
+
o(\frac{1}{\lambda^{2}}) \\
= & 
K_{j}\alpha_{j}^{p-1}\int_{\{\varphi_{j}>\sum_{j\neq l}\alpha^{l}\varphi_{l}\}} \varphi_{j}^{p-1}\tilde\nu_{i}\tilde\phi_{j, k}d\mu_{g_{0}}\\
& +
O
(
\frac{1}{\lambda}
\sum_{j\neq l}
\Vert \varphi_{l}^{p-1}\varphi_{j}
+
\varphi_{l}\varphi_{j}^{p-1}\Vert_{L^{\frac{p+1}{p}}}
)
+
o(\frac{1}{\lambda^{2}}) 
\end{split}
\end{equation} 
and therefore using Lemma \ref{lem_interactions} with 
$ p = \frac{n+2}{n-2} - \tau $ 
\begin{equation*}\begin{split}
\int K(\alpha^{l} \varphi_{l})^{p-1}\tilde\nu_{i}\tilde\phi_{j, k}d\mu_{g_{0}}
= 
K_{j}\alpha_{j}^{p-1}\int \varphi_{j}^{p-1}\tilde\nu_{i}\tilde\phi_{j, k}d\mu_{g_{0}}
+
o(\frac{1}{\lambda^{2}}). 
\end{split}
\end{equation*} 
Then, since $ \Vert \tilde \nu_{i}\Vert=O(\frac{1}{\lambda}), \tau=O(\frac{1}{\lambda^{2}}) $ and $ \varepsilon_{r, s}=O(\frac{1}{\lambda^{n-2}}) $, we find
\begin{equation*}\begin{split}
\int K(\alpha^{l} \varphi_{l})^{p-1} \tilde \nu_{i}\tilde\phi_{j, k}d\mu_{g_{0}}
= &
K_{j}\alpha_{j}^{\frac{4}{n-2}}\int \varphi_{j}^{\frac{4}{n-2}} \tilde \nu_{i}\tilde\phi_{j, k}d\mu_{g_{0}}
+
o(\frac{1}{\lambda^{2}}) 
=
o(\frac{1}{\lambda^{2}}), 
\end{split} 
\end{equation*} 
where the last inequality follows from 
Lemma \ref{lem_emergence_of_the_regular_part} 
and $ \langle \phi_{k, j}, \tilde \nu_{i}\rangle_{L_{g_{0}}}=0 $. Thus
\begin{equation}\label{second_variation_nu_phi_2_phi_3_interaction}
\partial^{2}J_{\tau}(\alpha^{l}\varphi_{l}+\bar v)
\tilde \nu_{i}\tilde\phi_{k, j}=o(\frac{1}{\lambda^{2}})
\; \text{ for }\; k=2, 3.
\end{equation}
By exactly the same argumentswith $ \tilde \phi_{1, i}=O(\frac{1}{\lambda}) $ as for \eqref{eq:crucial} there holds
\begin{equation*}
\partial^{2}J_{\tau}(\alpha^{l}\varphi_{l}+\bar v)\tilde \phi_{1, i}\tilde \phi_{k, j}
=
\partial^{2}J_{\tau}(\alpha^{l}\varphi_{l}+\bar v) \frac{\phi_{1, i}}{\lambda} \phi_{k, j}
=
\frac{1}{\lambda}\partial^{2}J_{\tau}(\alpha^{l}\varphi_{l}) \varphi_{i} \phi_{k, j}
=
o(\frac{1}{\lambda^{2}})
\end{equation*}
for $ k=2, 3 $. Thus we arrive at
\begin{equation*} 
[\partial^{2}J_{\tau}(\alpha^{l}\varphi_{l}+\bar v)]_{\mathbb{B}} 
=
 \begin{pmatrix}
\frac{1}{\lambda^2} \mathbb{V}_{+} & 0 & 0 & 0 \\
0 &\frac{1}{\lambda^2} \mathbb{A}_{q-1, 0} & 0 & 0\\
0 &0 & \partial^{2}J_{\tau} \tilde \phi_{2} \tilde \phi_{2} & \partial^{2}J_{\tau} \tilde \phi_{2}\tilde \phi_{3} \\
0 & 0 & \partial^{2}J_{\tau} \tilde \phi_{3}\tilde \phi_{2} & \partial^{2}J_{\tau} \tilde \phi_{3}, \tilde \phi_{3} \\
\end{pmatrix}
+o(\frac{1}{\lambda^{2}}).
\end{equation*}

\

\noindent {\bf{Last 2x2 block.}}We are left with the estimate of
 $$ \partial^{2}J_{\tau}(\alpha^{k}\varphi_{k}+\bar v) \tilde \phi_{k, i} \tilde \phi_{l, j}= \partial^{2}J_{\tau}(\alpha^{k}\varphi_{k}+\bar v) \phi_{k, i}\phi_{l, j} $$ 
for $ k, l=2, 3 $. 
Using the fact that 
\begin{equation*}
\int \phi_{k, i} L_{g_{0}}(\alpha^{m}\varphi_{m}+\bar v) d\mu_{g_{0}} 
=
o(\frac{1}{\lambda})
=
\int \phi_{k, i} K(\alpha^{m}\varphi_{m}+\bar v)^{p} d\mu_{g_{0}}
\; \text{ for }\; k=2, 3, 
\end{equation*}
which follows from $ \Vert\bar v\Vert =O(\frac{1}{\lambda^{2}}) $, Lemma \ref{lem_emergence_of_the_regular_part} and Lemma 
\ref{lem_interactions}, wefind for \eqref{second_variation_evaluating}
\begin{equation}\begin{split}\label{second_variation_split_I_1_I_2}
\partial^{2} & J_{\tau}(\alpha^{m}\varphi_{m} + \bar v)\phi_{k, i}\phi_{l, j} \\
= &
\frac{2}{(k_{\tau})_{\alpha^{m}\varphi_{m}+\bar v}^{\frac{2}{p+1}}}
\int 
[ 
\phi_{k, i} L_{g_{0}} \phi_{l, j}- 
p\frac{r_{\alpha^{m}\varphi_{m}+\bar v}}{(k_{\tau})_{\alpha^{m}\varphi_{m}+\bar v}} 
K(\alpha^{m}\varphi_{m}+\bar v)^{p-1}\phi_{k, i}\phi_{l, j}
] 
d\mu_{g_{0}} 
\\
=: &
\frac{2I}{(k_{\tau})_{\alpha^{m}\varphi_{m}+\bar v}^{\frac{2}{p+1}}} 
=:
\frac{2(I_{1}-I_{2})}{(k_{\tau})_{\alpha^{m}\varphi_{m}+\bar v}^{\frac{2}{p+1}}} 
= 
\frac{2}{(c_{0}\alpha_{K, \tau}^{\frac{2n}{n-2}})^{\frac{n-2}{n}}}
(I_{1}-I_{2}) + o\left( \frac{1}{\lambda^2} \right). 
\end{split}
\end{equation}
In the latter formula, recalling \eqref{eq:kp} and \eqref{def_refined_neighbourhood}, 
we have used the fact that 
 $$ (k_{\tau})_{\alpha^{m}\varphi_{m}+\bar v}^{\frac{2}{p+1}} = 
(c_{0}\alpha_{K, \tau}^{\frac{2n}{n-2}})^{\frac{n-2}{n}} + o(1)
 $$ 
and that both $ I_1, I_2 =O( \frac{1}{\lambda^2})$. 
Let us first compute $ I_{2} $, for which we clearly have
\begin{equation*} 
\begin{split}
I_{2}
= &
p\frac{r_{\alpha^{m}\varphi_{m}+\bar v}}{(k_{\tau})_{\alpha^{m}\varphi_{m}+\bar v}} 
\int K(\alpha^{m}\varphi_{m})^{p-1}\phi_{k, i}\phi_{l, j}d\mu_{g_{0}}
\\
& +
p(p-1)\frac{r_{\alpha^{m}\varphi_{m}+\bar v}}{(k_{\tau})_{\alpha^{m}\varphi_{m}+\bar v}} 
\int K(\alpha^{m}\varphi_{m})^{p-2}\phi_{k, i}\phi_{l, j}
\bar v d\mu_{g_{0}} 
\end{split}
\end{equation*}
up to an error $ o(\frac{1}{\lambda^{2}}) $, as $ \Vert \bar v \Vert=O(\frac{1}{\lambda^{2}}) $, and therefore still up to an error $ o(\frac{1}{\lambda^{2}}) $ 
\begin{equation*}
\begin{split}
I_{2}
= &
p\frac{r_{\alpha^{m}\varphi_{m}+\bar v}}{(k_{\tau})_{\alpha^{m}\varphi_{m}+\bar v}} 
\int K(\alpha^{m}\varphi_{m})^{p-1}\phi_{k, i}\phi_{l, j}d\mu_{g_{0}}
\\
& +
4n(n-1)\frac{n+2}{n-2}\frac{4}{n-2}\frac{\alpha^{2}}{\alpha_{K, \tau}^{\frac{2n}{n-2}}} 
\int K(\alpha^{m}\varphi_{m})^{\frac{6-n}{n-2}}\phi_{k, i}\phi_{l, j}
\bar v d\mu_{g_{0}}.
\end{split} 
\end{equation*}
As due to $ d(a_i, a_j) \simeq 1 $ for $ i \neq j $ the interactions terms in 
\eqref{eq:eijm} are of order 
$$ 
\varepsilon_{i, j} 
=
O(\frac{1}{\lambda^{n-2}})
=
o(\frac{1}{\lambda^{2}}),$$
we find
\begin{equation*}
\begin{split}
I_{2}
= &
p\frac{r_{\alpha^{m}\varphi_{m}+\bar v}}{(k_{\tau})_{\alpha^{m}\varphi_{m}+\bar v}}\delta_{i, j} \alpha_{i}^{p-1}
\int K\varphi_{i}^{p-1}\phi_{k, i}\phi_{l, i}d\mu_{g_{0}}\\
& +
4n(n-1)\frac{n+2}{n-2}\frac{4}{n-2}\frac{\alpha^{2}}{\alpha_{K, \tau}^{\frac{2n}{n-2}}} \delta_{i, j}\alpha_{i}^{\frac{6-n}{n-2}}
\int K\varphi_{i}^{\frac{6-n}{n-2}}\phi_{k, i}\phi_{l, i}
\bar v d\mu_{g_{0}} 
\end{split}
\end{equation*}
up to an error $ o(\frac{1}{\lambda^{2}}) $. 
Up to the 
same error we may simplify this using\eqref{def_refined_neighbourhood} to 
\begin{equation*}
\begin{split}
I_{2}
= &
p\frac{r_{\alpha^{m}\varphi_{m}+\bar v}}{(k_{\tau})_{\alpha^{m}\varphi_{m}+\bar v}} \delta_{i, j}K_{i}\alpha_{i}^{p-1}
\int \varphi_{i}^{p-1}\phi_{k, i}\phi_{l, i}d\mu_{g_{0}} \\
& +
4n(n-1)\frac{n+2}{n-2}\delta_{i, j}
\underset{B_{\varepsilon}(a_{i})}{\int} \frac{\nabla^{2}K_{i}}{2K_{i}}x^{2}\varphi_{i}^{\frac{4}{n-2}}\phi_{k, i}\phi_{l, i}d\mu_{g_{0}}
\\
& +
4n(n-1)\frac{n+2}{n-2}\frac{4}{n-2}\delta_{i, j}
\alpha_{i}^{-1}
\int \varphi_{i}^{\frac{6-n}{n-2}}\phi_{k, i}\phi_{l, i}
\bar v d\mu_{g_{0}} 
\end{split}
\end{equation*}
for some $ \varepsilon>0 $ small and fixed. 
Moreover by orthogonality and \eqref{r/lambda_expansion}
\begin{equation*}
\frac{r_{\alpha^{i}\varphi_{i}+\bar v}}{(k_{\tau})_{\alpha^{i}\varphi_{i}+\bar v}}
=
\frac{r_{\alpha^{i}\varphi_{i}}}{(k_{\tau})_{\alpha^{i}\varphi_{i}}}
=
4n(n-1)\frac{\alpha^{2}}{\alpha_{K, \theta}^{p+1}}
(
1
-
(\frac{\bar c_{1}}{\bar c_{0}}-\frac{\tilde c_{1}}{\tilde c_{2}}\frac{\bar c_{2}}{ \bar c_{0}})\tau
)
+
o(\frac{1}{\lambda^2}), 
\end{equation*}
whence by\eqref{def_refined_neighbourhood} and the fact that $ p=\frac{n+2}{n-2}-\tau $ we arrive at
\begin{equation*}
\begin{split}
I_{2}
= &
4n(n-1)\frac{n+2}{n-2}
[
(
1
-
(\frac{n-2}{n+2}+\frac{\bar c_{1}}{\bar c_{0}}-\frac{\tilde c_{1}}{\tilde c_{2}}\frac{\bar c_{2}}{ \bar c_{0}})\tau
)
]
\lambda_{i}^{\theta}\delta_{i, j}
\int \varphi_{i}^{p-1}\phi_{k, i}\phi_{l, i}d\mu_{g_{0}} \\
& +
4n(n-1)\frac{n+2}{n-2}\delta_{i, j} 
\underset{B_{\varepsilon}(a_{i})}{\int}\frac{\nabla^{2}K_{i}}{2K_{i}}x^{2}\varphi_{i}^{\frac{4}{n-2}}\phi_{k, i}\phi_{l, i}d\mu_{g_{0}}
\\
& +
4n(n-1)\frac{n+2}{n-2}\frac{4}{n-2}\delta_{i, j}
\alpha_{i}^{-1}
\int \varphi_{i}^{\frac{6-n}{n-2}}\phi_{k, i}\phi_{l, i}
\bar v d\mu_{g_{0}}.
\end{split}
\end{equation*}
Let us compute the last integral above, which is of order $ O(\frac{1}{\lambda^{2}}) $ as $ \Vert \bar v \Vert $. Clearly
\begin{equation*}
\begin{split}
\frac{4}{n-2}\int & \varphi_{i}^{\frac{6-n}{n-2}}\phi_{k, i}\phi_{l, i}
\bar v d\mu_{g_{0}}
= 
\int d_{k, i}\varphi_{i}^{\frac{4}{n-2}}\phi_{l, i}
\bar v d\mu_{g_{0}} \\
= &
d_{k, i} \int \varphi_{i}^{\frac{4}{n-2}}\phi_{l, i}
\bar v d\mu_{g_{0}}
-
\int \varphi_{i}^{\frac{4}{n-2}}d_{k, i}\phi_{l, i}
\bar v d\mu_{g_{0}}
-
\int \varphi_{i}^{\frac{4}{n-2}}\phi_{l, i}
d_{k, i} \bar v d\mu_{g_{0}}. 
\end{split}
\end{equation*}
Due to orthogonality the first integral above is of order $ o(\frac{1}{\lambda^{2}}) $ 
and denoting by 
\begin{equation}\label{eq:hatw}
\widehat w=\Pi_{\langle \phi_{k, i}\rangle^{\perp_{L_{g_{0}}}}}w\; \text{ for }\;w\in W^{1, 2}(M, g_0)
\end{equation}
the orthogonal projection onto $ \langle \phi_{k, i}\rangle^{\perp_{L_{g_{0}}}} $ 
we have up to an error $ o(\frac{1}{\lambda^{2}}) $ 
\begin{equation*}
\begin{split}
\int \varphi_{i}^{\frac{4}{n-2}}d_{k, i}\phi_{l, i}
\bar v d\mu_{g_{0}}
= &
\int \varphi_{i}^{\frac{4}{n-2}}\widehat{d_{k, i}\phi_{l, i}} 
\bar v d\mu_{g_{0}}
\end{split}
\end{equation*}
due to the orthogonalities $ \langle \bar v, \phi_{k, i}\rangle_{L_{g_{0}}}=0 $ and 
 $ \Vert \bar v\Vert=O(\frac{1}{\lambda^{2}}) $. Hence, 
using the notation in \eqref{eq:hatw}, we arrive at
\begin{equation*}
\begin{split}
I_{2}
= &
4n(n-1)\frac{n+2}{n-2}
[
(
1
-
(\frac{n-2}{n+2}+\frac{\bar c_{1}}{\bar c_{0}}-\frac{\tilde c_{1}}{\tilde c_{2}}\frac{\bar c_{2}}{ \bar c_{0}})\tau
)
]
\lambda_{i}^{\theta}\delta_{i, j}
\int \varphi_{i}^{p-1}\phi_{k, i}\phi_{l, i}d\mu_{g_{0}} \\
& +
4n(n-1)\frac{n+2}{n-2}\delta_{i, j}
\int_{B_{c}(a_{i})} \frac{\nabla^{2}K_{i}}{2K_{i}}x^{2}\varphi_{i}^{\frac{4}{n-2}}\phi_{k, i}\phi_{l, i}d\mu_{g_{0}}
\\
& -
4n(n-1)\frac{n+2}{n-2}\delta_{i, j}
\alpha_{i}^{-1}
(
\int \varphi_{i}^{\frac{4}{n-2}}\widehat{d_{k, i}\phi_{l, i}}
\bar v d\mu_{g_{0}}
+
\int \varphi_{i}^{\frac{4}{n-2}}\phi_{l, i}
d_{k, i} \bar v d\mu_{g_{0}}
).
\end{split}
\end{equation*}
Due to $ \Vert \bar v \Vert=O(\frac{1}{\lambda^{2}}) $ we havestill up to a $ o(\frac{1}{\lambda^{2}}) $ 
\begin{equation*}
\begin{split}
\partial^{2}J_{\tau}(\alpha^{m}\varphi_{m})\bar v
= &
\frac{8n(n-1)}{(\bar c_{0}\alpha_{K, \tau}^{p+1})^{\frac{n-2}{n}}}
\left( \frac{L_{g_{0}}}{4n(n-1)}\bar v
-
\frac{n+2}{n-2}\sum_{m} 
\varphi_{m}^{\frac{4}{n-2}}\bar v
\right)
\end{split}
\end{equation*}
and we recall from \eqref{bar_v_solution_to} that 
\begin{equation*}
\partial^{2}J_{\tau}(\alpha^{m}\varphi_{m})\bar v
=
-\partial J_{\tau}(\alpha^{m}\varphi_{m})+o(\frac{1}{\lambda^{2}})
\; \text{ on }\; \langle \phi_{l, j}\rangle^{\perp_{L_{g_{0}}}}.
\end{equation*}
From this we deduce again by smallness of interactions terms $ \varepsilon_{i, j} $ 
\begin{equation*}
\begin{split}
\frac{n+2}{n-2}
\int \varphi_{i}^{\frac{4}{n-2}}\widehat{d_{k, i}\phi_{l, i}}
\bar v d\mu_{g_{0}}
= &
\frac{(\bar c_{0}\alpha_{K, \tau}^{p+1})^{\frac{n-2}{n}}}{8n(n-1)}\partial J_{\tau}(\alpha^{m}\varphi_{m})
\widehat{d_{k, i}\phi_{l, i}} 
+ 
 \frac{\langle \bar v, \widehat{d_{k, i}\phi_{l, i}}\rangle_{L_{g_{0}}} }{4n(n-1)}
\end{split}
\end{equation*}
and by orthogonality and Lemma \ref{lem_emergence_of_the_regular_part} there holds 
up to an error $ o(\frac{1}{\lambda^{2}}) $ 
\begin{equation*}
\begin{split}
\langle \bar v, \widehat{d_{k, i}\phi_{l, i}}\rangle_{L_{g_{0}}} 
= &
-
\langle d_{k, i} \bar v, \phi_{l, i}\rangle_{L_{g_{0}}} 
= 
-4n(n-1)\int \bar d_{k, i} v d_{l, i}\varphi^{\frac{n+2}{n-2}} d\mu_{g_{0}} \\
= &
-4n(n-1)\frac{n+2}{n-2}\int \bar \varphi_{i}^{\frac{4}{n-2}}d_{k, i} v \phi_{l, i}d\mu_{g_{0}}.
\end{split}
\end{equation*}
We therefore conclude that up to an error $ o(\frac{1}{\lambda^2}) $ 
\begin{equation*}
\begin{split}
I_{2}
= &
4n(n-1)\frac{n+2}{n-2}
[
(
1
-
(\frac{n-2}{n+2}+\frac{\bar c_{1}}{\bar c_{0}}-\frac{\tilde c_{1}}{\tilde c_{2}}\frac{\bar c_{2}}{ \bar c_{0}})\tau
)
]
\lambda_{i}^{\theta}\delta_{i, j}
\int \varphi_{i}^{p-1}\phi_{k, i}\phi_{l, i}d\mu_{g_{0}} \\
& +
4n(n-1)\frac{n+2}{n-2}\delta_{i, j} 
\underset{B_{\varepsilon}(a_{i})}{\int} \frac{\nabla^{2}K_{i}}{2K_{i}}x^{2}\varphi_{i}^{\frac{4}{n-2}}\phi_{k, i}\phi_{l, i}d\mu_{g_{0}} \\
&- 
4n(n-1)\delta_{i, j}
\alpha_{i}^{-1}
\frac{(\bar c_{0}\alpha_{K, \tau}^{p+1})^{\frac{n-2}{n}}}{8n(n-1)}\partial J_{\tau}(\alpha^{m}\varphi_{m})
\widehat{d_{k, i}\phi_{l, i}}, 
\end{split}
\end{equation*}
at which point $ \bar v $ has been eliminated from the main terms in the expansion. By Lemma \ref{lem_refined_gradient} we then have 
\begin{equation*}
\partial J_{\tau}(\alpha^{m}\varphi_{m})\lfloor_{\langle \phi_{k, i}\rangle }=o(\frac{1}{\lambda^{2}}), 
\end{equation*}
so we may pass from $ \widehat{d_{k, i}\phi_{l, i}} $ to $ d_{k, i}\phi_{l, i} $ in the above formulae and, 
as 
$$ \partial J_{\tau}(\alpha^{m}\varphi_{m})=O(\frac{1}{\lambda^{2}}),$$ 
we obtain
\begin{equation*}
\begin{split}
 \frac{(\bar c_{0}\alpha_{K, \tau}^{p+1})^{\frac{n-2}{n}}}{8n(n-1)}&
\partial J_{\tau}(\alpha^{m}\varphi_{m})d_{k, i}\phi_{l, i}\\
= &
-
\alpha^{m} 
\tau 
\underset{B_{\varepsilon}(a_{m})}{\int}
\biggr(
\varphi_{m}^{\frac{n+2}{n-2}}\ln(1+\lambda_{m}^{2}r^{2})^{\frac{n-2}{2}}
-
\frac{\bar c_{1}}{c_{1}}\varphi_{m}^{\frac{n+2}{n-2}} \\
& \quad\quad\quad\quad\quad\quad\quad\quad +
\frac{2}{n-2}\frac{\tilde c_{1}}{c_{2}}
\varphi_{m}^{\frac{4}{n-2}}\lambda_{m}\partial_{\lambda_{m}}\varphi_{m} 
\biggr)
d_{k, i}\phi_{l, i}
d\mu_{g_{0}} 
\\
&
+
\alpha^{m}
\tau 
\underset{B_{\varepsilon}(a_{m})}{\int}
\biggr(
\frac{\tilde c_{1}}{\tilde c_{2}}\frac{ \lambda_{m}^{2}r^{2}}{2n}\varphi_{m}^{\frac{n+2}{n-2}}
 -
\frac{\tilde c_{1}\bar c_{2}}{\tilde c_{2}c_{1}}\varphi_{m}^{\frac{n+2}{n-2}} \\
& \quad\quad\quad\quad\quad\quad\quad\quad +
\frac{2}{n-2}\frac{\tilde c_{1}}{c_{2}}\varphi_{m}^{\frac{4}{n-2}}\lambda_{m}\partial_{\lambda_{m}}\varphi_{m}
\biggr)d_{k, i}\phi_{l, i} d\mu_{g_{0}}
\\
& -
\alpha^{m}
\underset{B_{\varepsilon}(a_{m})}{\int}
(
\frac{\nabla^{2}K_{m}}{2K_{m}}
x^{2}
-
\frac{\Delta K_{m}}{2nK_{m}}r^{2}
)
\varphi_{m}^{\frac{n+2}{n-2}}d_{k, i}\phi_{l, i} d\mu_{g_{0}}. 
\end{split} 
\end{equation*}
Still $ \varepsilon_{i, j} = o(\frac{1}{\lambda^2}) $ we therefore arrive at 
\begin{equation*}
\begin{split}
I_{2}
= &
4n(n-1)\frac{n+2}{n-2}
[
(
1
-
(\frac{n-2}{n+2}+\frac{\bar c_{1}}{\bar c_{0}}-\frac{\tilde c_{1}}{\tilde c_{2}}\frac{\bar c_{2}}{ \bar c_{0}})\tau
)
]
\lambda_{i}^{\theta}\delta_{i, j}
\int \varphi_{i}^{p-1}\phi_{k, i}\phi_{l, i}d\mu_{g_{0}} \\
& +
4n(n-1)\frac{n+2}{n-2}\delta_{i, j} 
\underset{B_{\varepsilon}(a_{i})}{\int} \frac{\nabla^{2}K_{i}}{2K_{i}}x^{2}\varphi_{i}^{\frac{4}{n-2}}\phi_{k, i}\phi_{l, i}d\mu_{g_{0}}
\\
& -
4n(n-1)\delta_{i, j}
\biggr(
-
\tau 
\underset{B_{\varepsilon}(a_{i})}{\int}
\biggr(
\varphi_{i}^{\frac{n+2}{n-2}}\ln(1+\lambda_{i}^{2}r^{2})^{\frac{n-2}{2}}
-
\frac{\bar c_{1}}{c_{1}}\varphi_{i}^{\frac{n+2}{n-2}} \\
& \quad\quad\quad\quad\quad\quad\quad\quad\quad\quad\quad\quad\quad\quad +
\frac{2}{n-2}\frac{\tilde c_{1}}{c_{2}}
\varphi_{i}^{\frac{4}{n-2}}\lambda_{i}\partial_{\lambda_{i}}\varphi_{i} 
\biggr)
d_{k, i}\phi_{l, i}
d\mu_{g_{0}} 
\\
&\quad \quad \quad \quad \quad \quad \quad \quad 
+
\tau 
\underset{B_{\varepsilon}(a_{i})}{\int}
\biggr(
\frac{\tilde c_{1}}{\tilde c_{2}}\frac{ \lambda_{i}^{2}r^{2}}{2n}\varphi_{i}^{\frac{n+2}{n-2}}
 -
\frac{\tilde c_{1}\bar c_{2}}{\tilde c_{2}c_{1}}\varphi_{i}^{\frac{n+2}{n-2}} \\
& \quad\quad\quad\quad\quad\quad\quad\quad\quad\quad\quad\quad\quad\quad +
\frac{2}{n-2}\frac{\tilde c_{1}}{c_{2}}\varphi_{i}^{\frac{4}{n-2}}\lambda_{i}\partial_{\lambda_{i}}\varphi_{i}
\biggr)d_{k, i}\phi_{l, i} d\mu_{g_{0}}
\\
& \quad \quad \quad \quad \quad \quad \quad \quad 
-
\underset{B_{\varepsilon}(a_{i})}{\int}
(
\frac{\nabla^{2}K_{i}}{2K_{i}}
x^{2}
-
\frac{\Delta K_{i}}{2nK_{i}}r^{2}
)
\varphi_{i}^{\frac{n+2}{n-2}}d_{k, i}\phi_{l, i} d\mu_{g_{0}}
\biggr), 
\end{split}
\end{equation*}
up to some $ o(\frac{1}{\lambda^{2}}) $.
By oddness we may simplify this to
\begin{equation*}
\begin{split}
I_{2}
= &
4n(n-1)\frac{n+2}{n-2}
[
(
1
-
(\frac{n-2}{n+2}+\frac{\bar c_{1}}{\bar c_{0}}-\frac{\tilde c_{1}}{\tilde c_{2}}\frac{\bar c_{2}}{ \bar c_{0}})\tau
)
]\\
& \quad\quad\quad\quad\quad
\lambda_{i}^{\theta}\delta_{i, j}\delta_{k, l}
\int \varphi_{i}^{p-1} \phi_{k, i} \phi_{k, i} d\mu_{g_{0}} \\
& +
4n(n-1)\frac{n+2}{n-2}\delta_{i, j} \delta_{k, l}
\underset{B_{\varepsilon}(a_{i})}{\int}
\frac{\nabla^{2}K_{i}}{2K_{i}}x^{2}\varphi_{i}^{\frac{4}{n-2}}\phi_{k, i} \phi_{k, i} d\mu_{g_{0}}
\\
& -
4n(n-1)\delta_{i, j}\delta_{k, l}
\biggr(
-
\tau 
\underset{B_{\varepsilon}(a_{i})}{\int}
\biggr(
\varphi_{i}^{\frac{n+2}{n-2}}\ln(1+\lambda_{i}^{2}r^{2})^{\frac{n-2}{2}}
-
\frac{\bar c_{1}}{c_{1}}\varphi_{i}^{\frac{n+2}{n-2}} \\
& \quad\quad\quad\quad\quad\quad\quad\quad\quad\quad\quad\quad\quad\quad\quad
+
\frac{2}{n-2}\frac{\tilde c_{1}}{c_{2}}
\varphi_{i}^{\frac{4}{n-2}}\lambda_{i}\partial_{\lambda_{i}}\varphi_{i} 
\biggr)
d_{k, i}\phi_{k, i}
d\mu_{g_{0}} 
\\
&\quad \quad \quad \quad \quad \quad \quad \quad \quad \;
+
\tau 
\underset{B_{\varepsilon}(a_{i})}{\int}
\biggr(
\frac{\tilde c_{1}}{\tilde c_{2}}\frac{ \lambda_{i}^{2}r^{2}}{2n}\varphi_{i}^{\frac{n+2}{n-2}}
 -
\frac{\tilde c_{1}\bar c_{2}}{\tilde c_{2}c_{1}}\varphi_{i}^{\frac{n+2}{n-2}} \\
&\quad\quad\quad\quad\quad\quad\quad\quad\quad\quad\quad\quad\quad\quad\quad +
\frac{2}{n-2}\frac{\tilde c_{1}}{c_{2}}\varphi_{i}^{\frac{4}{n-2}}\lambda_{i}\partial_{\lambda_{i}}\varphi_{i}
\biggr)d_{k, i}\phi_{k, i} d\mu_{g_{0}}
\\
& \quad \quad \quad \quad \quad \quad \quad \quad \quad \;
-
\underset{B_{\varepsilon}(a_{i})}{\int}
(
\frac{\nabla^{2}K_{i}}{2K_{i}}
x^{2}
-
\frac{\Delta K_{i}}{2nK_{i}}r^{2}
)
\varphi_{i}^{\frac{n+2}{n-2}}d_{k, i}\phi_{k, i} d\mu_{g_{0}}
\biggr)
\end{split}
\end{equation*}
 By Lemma \ref{lem_emergence_of_the_regular_part}it follows that 
 for $ k=2, 3 $ andup to some $ o(\frac{1}{\lambda^{2}}) $ 
\begin{equation*}
\begin{split}
4n & (n-1) \frac{n+2}{n-2}\int \varphi_{i}^{\frac{4}{n-2}}\lambda_{i}\partial_{\lambda_{i}}\varphi_{i}d_{k, i}\phi_{k, i}d\mu_{g_{0}}
= 
\int L_{g_{0}}(\lambda_{i}\partial_{\lambda_{i}}\varphi_{i})d_{k, i}\phi_{k, i}d\mu_{g_{0}} \\
= &
\langle \lambda_{i}\partial_{\lambda_{i}}\varphi_{i}, (d_{k, i})^{2}\varphi_{i}\rangle_{L_{g_{0}}}
=
d_{k, i}\langle \lambda_{i}\partial_{\lambda_{i}}\varphi_{i}, d_{k, i}\varphi_{i}\rangle_{L_{g_{0}}}
-
\langle \lambda_{i}\partial_{\lambda_{i}}d_{k, i}\varphi_{i}, d_{k, i}\varphi_{i}\rangle_{L_{g_{0}}} \\
= &
d_{k, i}\langle \phi_{2, i}, \phi_{k, i}\rangle_{L_{g_{0}}}
-
\frac{1}{2}
\lambda_{i}\partial_{\lambda_{i}}
\Vert \phi_{k, i}^{2}\Vert_{L_{g_{0}}}=o(1), 
\end{split}
\end{equation*}
as $ \langle \phi_{2, i}, \phi_{k, i}\rangle_{L_{g_{0}}} $ and $ \Vert \phi_{k, i}^{2}\Vert_{L_{g_{0}}} $ are almost constant in $ a_{i} $ and $ \lambda_{i} $. 
Hence
\begin{equation*}
\begin{split}
\frac{I_{2}}{4n(n-1)}
= &
\frac{n+2}{n-2}
[
(
1
-
(\frac{n-2}{n+2}+\frac{\bar c_{1}}{\bar c_{0}}-\frac{\tilde c_{1}}{\tilde c_{2}}\frac{\bar c_{2}}{ \bar c_{0}})\tau
)
]
\lambda_{i}^{\theta}\delta_{i, j}\delta_{k, l}
\int \varphi_{i}^{p-1}\phi_{k, i} \phi_{k, i}d\mu_{g_{0}} \\
& +
\frac{n+2}{n-2}\delta_{i, j} \delta_{k, l}
\underset{B_{\varepsilon}(a_{i})}{\int}\frac{\nabla^{2}K_{i}}{2K_{i}}x^{2}\varphi_{i}^{\frac{4}{n-2}}\phi_{k, i} \phi_{k, i} d\mu_{g_{0}} \\
& -
\delta_{i, j}\delta_{k, l}
\biggr(
-
\tau 
\underset{B_{\varepsilon}(a_{i})}{\int}
\biggr(
\ln(1+\lambda_{i}^{2}r^{2})^{\frac{n-2}{2}}
-
\frac{\bar c_{1}}{c_{1}}
\biggr)
\varphi_{i}^{\frac{n+2}{n-2}}d_{k, i}\phi_{k, i}
d\mu_{g_{0}} 
\\
&\quad \quad \quad \quad \quad
+
\tau 
\underset{B_{\varepsilon}(a_{i})}{\int}
\biggr(
\frac{\tilde c_{1}}{\tilde c_{2}}\frac{ \lambda_{i}^{2}r^{2}}{2n}
 -
\frac{\tilde c_{1}\bar c_{2}}{\tilde c_{2}c_{1}}
\biggr) \varphi_{i}^{\frac{n+2}{n-2}}d_{k, i}\phi_{k, i} d\mu_{g_{0}} \\
& \quad \quad \quad \quad \quad -
\underset{B_{\varepsilon}(a_{i})}{\int}^
(
\frac{\nabla^{2}K_{i}}{2K_{i}}
x^{2}
-
\frac{\Delta K_{i}}{2nK_{i}}r^{2}
)
\varphi_{i}^{\frac{n+2}{n-2}}d_{k, i}\phi_{k, i} d\mu_{g_{0}}
\biggr).
\end{split}
\end{equation*}
Next for the first summand above we find that up to an error $ o(\frac{1}{\lambda^{2}}) $ 
\begin{equation*}
\begin{split}
\lambda_{i}^{\theta}\int& \varphi_{i}^{p-1}\phi_{k, i} \phi_{k, i} 
d\mu_{g_{0}} \\
& =
\int \varphi_{i}^{\frac{4}{n-2}}\phi_{k, i} \phi_{k, i}d\mu_{g_{0}}
+
\underset{B_{\varepsilon}(a_{i})}{\int}
\varphi_{i}^{\frac{4}{n-2}}( \lambda_{i}^{\theta}\varphi_{i}^{-\tau}-1)\phi_{k, i} \phi_{k, i}d\mu_{g_{0}} \\
= &
\frac{n-2}{n+2}
\int d_{k, i}\varphi_{i}^{\frac{n+2}{n-2}}\phi_{k, i}d\mu_{g_{0}}
+
\underset{B_{\varepsilon}(a_{i})}{\int}
\varphi_{i}^{\frac{4}{n-2}}((1+\lambda_{i}^{2}r^{2})^{\theta}-1)\phi_{k, i} \phi_{k, i}d\mu_{g_{0}} \\
= &
\frac{1}{4n(n-1)}\frac{n-2}{n+2}\langle \phi_{k, i}, \phi_{k, i}\rangle_{ L_{g_{0}}}
+
\theta \underset{B_{\varepsilon}(a_{i})}{\int}\varphi_{i}^{\frac{4}{n-2}}\ln(1+\lambda_{i}^{2}r^{2})
\phi_{k, i} \phi_{k, i} d\mu_{g_{0}}
\end{split}
\end{equation*}
using Lemma \ref{lem_emergence_of_the_regular_part} and properly expanding. Recalling \eqref{second_variation_split_I_1_I_2} we thus conclude
\begin{equation}\begin{split}\label{second_variation_nu_phi_2_phi_3_interaction_rough}
 \frac{(k_{\tau})_{\alpha^{m}\varphi_{m}+\bar v}^{\frac{2}{p+1}}}{8n(n-1)}
\hspace {-20pt} & \hspace {20pt}
\partial^{2}J_{\tau}(\alpha^{m}\varphi_{m}+\bar v)\phi_{k, i}\phi_{l, j} \\
= &
\int \frac{L_{g_{0}}}{4n(n-1)}\phi_{k, i}\phi_{l, j} d\mu_{g_{0}} 
- 
\frac{I_{2}}{4n(n-1)} \\
= &
\delta_{i, j}\delta_{k, l}
\biggr(
(
1+\frac{n+2}{n-2}(\frac{\bar c_{1}}{\bar c_{0}}-\frac{\tilde c_{1}}{\tilde c_{2}}\frac{\bar c_{2}}{ \bar c_{0}})
)
\tau
\int \varphi_{i}^{\frac{4}{n-2}}\phi_{k, i} \phi_{k, i}d\mu_{g_{0}} \\
& \quad \quad \quad \quad \quad -
\frac{n+2}{n-2}
 \tau 
\underset{B_{\varepsilon}(a_{i})}{\int}\varphi_{i}^{\frac{4}{n-2}}\ln(1+\lambda_{i}^{2}r^{2})^{\frac{n-2}{2}}
\phi_{k, i} \phi_{k, i} d\mu_{g_{0}}\\
& \quad \quad \quad \quad \quad -
\tau 
\underset{B_{\varepsilon}(a_{i})}{\int}
\biggr(
\ln(1+\lambda_{i}^{2}r^{2})^{\frac{n-2}{2}}
-
\frac{\bar c_{1}}{c_{1}}
\biggr)
\varphi_{i}^{\frac{n+2}{n-2}}d_{k, i}\phi_{k, i}
d\mu_{g_{0}} \\
& \quad \quad \quad \quad \quad +
\tau 
\underset{B_{\varepsilon}(a_{i})}{\int}
\biggr(
\frac{\tilde c_{1}}{\tilde c_{2}}\frac{ \lambda_{i}^{2}r^{2}}{2n}
 -
\frac{\tilde c_{1}\bar c_{2}}{\tilde c_{2}c_{1}}
\biggr) \varphi_{i}^{\frac{n+2}{n-2}}d_{k, i}\phi_{k, i} d\mu_{g_{0}}
\\
& \quad \quad \quad \quad \quad 
-
\underset{B_{\varepsilon}(a_{i})}{\int}
(
\frac{\nabla^{2}K_{i}}{2K_{i}}
x^{2}
-
\frac{\Delta K_{i}}{2nK_{i}}r^{2}
)
\varphi_{i}^{\frac{n+2}{n-2}}d_{k, i}\phi_{k, i} d\mu_{g_{0}} \\
&\quad \quad \quad \quad \quad -
\frac{n+2}{n-2}
\underset{B_{\varepsilon}(a_{i})}{\int} \frac{\nabla^{2}K_{i}}{2K_{i}}x^{2}\varphi_{i}^{\frac{4}{n-2}}\phi_{k, i} \phi_{k, i} d\mu_{g_{0}}
\biggr)
\end{split}
\end{equation}
and in particular for $ i = 1, \dots, q $, and $ k, l = 1, \dots, n $ 
we have up some $ o(\frac{1}{\lambda^{2}}) $ 
\begin{equation*} 
[\partial^{2}J_{\tau}(\alpha^{k}\varphi_{k}+\bar v)]_{\mathbb{B}} 
=
 \begin{pmatrix}
\frac{1}{\lambda^2} \mathbb{V}_{+} & 0 & 0 & 0 \\
0 &\frac{1}{\lambda^2} \mathbb{A}_{q-1, 0} & 0 & 0\\
0 &0 & \partial^{2}J_{\tau} \lambda_{i} \partial_{\lambda_{i}}\varphi_{i} \lambda_{i}\partial_{\lambda_{i}}\varphi_{i} & 0 \\
0 & 0 & 0 &
\partial^{2}J_{\tau} \frac{(\nabla_{a_{i}})_{k}}{\lambda_{i}}\varphi_{i} \frac{(\nabla_{a_{i}})_{l}}{\lambda_{i}}\varphi_{i}\\
\end{pmatrix}. 
\end{equation*}

\

\noindent {\bf{Last diagonal terms.}} 
Concerning $ \lambda $-derivatives we note, that mixed derivatives in different $ \lambda_{i} $   are 
of order $ \lambda^{2-n}= o(\lambda^{-2})$, since $ n \geq 5 $. Therefore 
it is sufficient to compute second derivatives with respect to the same $ \lambda_{i} $. 
This corresponds to 
\begin{equation*}\begin{split}
\frac{(k_{\tau})_{\alpha^{m}\varphi_{m}+\bar v}^{\frac{2}{p+1}}}{8n(n-1)}
&
\partial^{2}J_{\tau}(\alpha^{m}\varphi_{m}+\bar v)(\lambda_{i}\partial_{\lambda_{i}}\varphi_{i})^{2} \\
= &
(
1+\frac{n+2}{n-2}(\frac{\bar c_{1}}{\bar c_{0}}-\frac{\tilde c_{1}}{\tilde c_{2}}\frac{\bar c_{2}}{ \bar c_{0}})
)
\tau
\int \varphi_{i}^{\frac{4}{n-2}}\phi_{k, i} \phi_{k, i}d\mu_{g_{0}}
\\
&
 -
\frac{n+2}{n-2}
 \tau 
\underset{B_{\varepsilon}(a_{i})}{\int}\varphi_{i}^{\frac{4}{n-2}}\ln(1+\lambda_{i}^{2}r^{2})^{\frac{n-2}{2}}
\vert\lambda_{i}\partial_{\lambda_{i}}\varphi_{i}\vert^{2}d\mu_{g_{0}}\\
&
-
\tau 
\underset{B_{\varepsilon}(a_{i})}{\int}
\biggr(
\ln(1+\lambda_{i}^{2}r^{2})^{\frac{n-2}{2}}
-
\frac{\bar c_{1}}{c_{1}}
\biggr)
\varphi_{i}^{\frac{n+2}{n-2}}(\lambda_{i}\partial_{\lambda_{i}})^{2}\varphi_{i}
d\mu_{g_{0}} \\
& +
\tau 
\underset{B_{\varepsilon}(a_{i})}{\int}
\biggr(
\frac{\tilde c_{1}}{\tilde c_{2}}\frac{ \lambda_{i}^{2}r^{2}}{2n}
 -
\frac{\tilde c_{1}\bar c_{2}}{\tilde c_{2}c_{1}}
\biggr) \varphi_{i}^{\frac{n+2}{n-2}}(\lambda_{i}\partial_{\lambda_{i}})^{2}\varphi_{i} d\mu_{g_{0}}
\\
&
-
\underset{B_{\varepsilon}(a_{i})}{\int}
(
\frac{\nabla^{2}K_{i}}{2K_{i}}
x^{2}
-
\frac{\Delta K_{i}}{2nK_{i}}r^{2}
)
\varphi_{i}^{\frac{n+2}{n-2}}(\lambda_{i}\partial_{\lambda_{i}})^{2}\varphi_{i} d\mu_{g_{0}} \\
& -
\frac{n+2}{n-2}
\underset{B_{\varepsilon}(a_{i})}{\int} \frac{\nabla^{2}K_{i}}{2K_{i}}x^{2}\varphi_{i}^{\frac{4}{n-2}}\vert \lambda_{i}\partial_{\lambda_{i}}\varphi_{i}\vert^{2}d\mu_{g_{0}}.
\end{split}
\end{equation*}
The second-last summand vanishes and 
$$ 
\int \varphi_{i}^{p-1}\phi_{k, i} \phi_{k, i}d\mu_{g_{0}} 
= 
c_{k} + o(1), 
$$
cf. Lemma \ref{lem_interactions}, whence
\begin{equation*}\begin{split}
\frac{(k_{\tau})_{\alpha^{m}\varphi_{m}+\bar v}^{\frac{2}{p+1}}}{8n(n-1)}
&
\partial^{2}J_{\tau}(\alpha^{m}\varphi_{m}+\bar v)(\lambda_{i}\partial_{\lambda_{i}}\varphi_{i})^{2} 
= 
c_{2}
(
1+\frac{n+2}{n-2}(\frac{\bar c_{1}}{\bar c_{0}}-\frac{\tilde c_{1}}{\tilde c_{2}}\frac{\bar c_{2}}{ \bar c_{0}})
)
\tau
\\
&
 -
\frac{n+2}{n-2}
 \tau 
\underset{B_{\varepsilon}(a_{i})}{\int}\varphi_{i}^{\frac{4}{n-2}}\ln(1+\lambda_{i}^{2}r^{2})^{\frac{n-2}{2}}
\vert\lambda_{i}\partial_{\lambda_{i}}\varphi_{i}\vert^{2}d\mu_{g_{0}} \\
& -
\tau 
\underset{B_{\varepsilon}(a_{i})}{\int}
\biggr(
\ln(1+\lambda_{i}^{2}r^{2})^{\frac{n-2}{2}}
-
\frac{\bar c_{1}}{c_{1}}
\biggr)
\varphi_{i}^{\frac{n+2}{n-2}}(\lambda_{i}\partial_{\lambda_{i}})^{2}\varphi_{i}
d\mu_{g_{0}} 
\\
&
+
\tau 
\underset{B_{\varepsilon}(a_{i})}{\int}
\biggr(
\frac{\tilde c_{1}}{\tilde c_{2}}\frac{ \lambda_{i}^{2}r^{2}}{2n}
 -
\frac{\tilde c_{1}\bar c_{2}}{\tilde c_{2}c_{1}}
\biggr) \varphi_{i}^{\frac{n+2}{n-2}}(\lambda_{i}\partial_{\lambda_{i}})^{2}\varphi_{i} d\mu_{g_{0}} \\
& -
\frac{n+2}{n-2}\frac{\Delta K_{i}}{2n K_{i}}
\underset{B_{\varepsilon}(a_{i})}{\int} r^{2}\varphi_{i}^{\frac{4}{n-2}}\vert \lambda_{i}\partial_{\lambda_{i}}\varphi_{i}\vert^{2}d\mu_{g_{0}}.
\end{split}
\end{equation*}
Moreover
\begin{equation*}
\begin{split}
\int \varphi_{i}^{\frac{n+2}{n-2}}
&
(\lambda_{i}\partial_{\lambda_{i}})^{2}\varphi_{i}d\mu_{g_{0}}
= 
\lambda_{i}\partial_{\lambda_{i}}\int \varphi_{i}^{\frac{n+2}{n-2}}\lambda_{i}\partial_{\lambda_{i}}\varphi_{i}d\mu_{g_{0}} \\
& -
\frac{n+2}{n-2}\int \varphi_{i}^{\frac{4}{n-2}}\vert \lambda_{i}\partial_{\lambda_{i}}\varphi_{i}\vert ^{2}d\mu_{g_{0}}
=
-\frac{n+2}{n-2}c_{2}
+
o(1), 
\end{split}
\end{equation*}
and 
\begin{equation*}
\begin{split}
\frac{n+2}{n-2}
\underset{B_{\varepsilon}(a_{i})}{\int} & r^{2}\varphi_{i}^{\frac{4}{n-2}}
\vert \lambda_{i}\partial_{\lambda_{i}}\varphi_{i}\vert^{2}d\mu_{g_{0}}
=
\underset{B_{\varepsilon}(a_{i})}{\int} r^{2}\lambda_{i}\partial_{\lambda_{i}}\varphi_{i}\lambda_{i}\partial_{\lambda_{i}}\varphi_{i}^{\frac{n+2}{n-2}}d\mu_{g_{0}} \\
= &
\lambda_{i}\partial_{\lambda_{i}}
\underset{B_{\varepsilon}(a_{i})}{\int} r^{2}\lambda_{i}\partial_{\lambda_{i}}\varphi_{i}\varphi_{i}^{\frac{n+2}{n-2}}d\mu_{g_{0}} 
-
\underset{B_{\varepsilon}(a_{i})}{\int} r^{2}\varphi_{i}^{\frac{n+2}{n-2}}(\lambda_{i}\partial_{\lambda_{i}})^{2}\varphi_{i}d\mu_{g_{0}}.
\end{split}
\end{equation*}
Thus recalling\eqref{def_refined_neighbourhood}, in particular 
 $ \tilde c_{1}\tau +\tilde c_{2}\frac{\Delta K_{i}}{K_{i}\lambda_{i}^{2}}=o(\frac{1}{\lambda^{2}}) $, 
we arrive at 
\begin{equation*}\begin{split}
\frac{(k_{\tau})_{\alpha^{m}\varphi_{m}+\bar v}^{\frac{2}{p+1}}}{8n(n-1)}
&
\partial^{2}J_{\tau}(\alpha^{m}\varphi_{m}+\bar v)(\lambda_{i}\partial_{\lambda_{i}}\varphi_{i})^{2} 
= 
c_{2}
\tau \\
& -
\frac{n+2}{n-2}
 \tau 
\underset{B_{\varepsilon}(a_{i})}{\int}\varphi_{i}^{\frac{4}{n-2}}\ln(1+\lambda_{i}^{2}r^{2})^{\frac{n-2}{2}}
\vert\lambda_{i}\partial_{\lambda_{i}}\varphi_{i}\vert^{2}d\mu_{g_{0}}\\
&
-
\tau 
\underset{B_{\varepsilon}(a_{i})}{\int}
\ln(1+\lambda_{i}^{2}r^{2})^{\frac{n-2}{2}}
\varphi_{i}^{\frac{n+2}{n-2}}(\lambda_{i}\partial_{\lambda_{i}})^{2}\varphi_{i}
d\mu_{g_{0}} \\
& + 
\frac{\tilde c_{1}}{\tilde c_{2}}\frac{\tau}{2n} 
\lambda_{i}^{3}\partial_{\lambda_{i}}
\underset{B_{\varepsilon}(a_{i})}{\int} r^{2}\lambda_{i}\partial_{\lambda_{i}}\varphi_{i}\varphi_{i}^{\frac{n+2}{n-2}}d\mu_{g_{0}}, 
\end{split}
\end{equation*}
and for the last integral above we find passing to integration over $ \R^{n} $ 
\begin{equation*}
\begin{split}
\lambda_{i}\partial_{\lambda_{i}} &
\underset{B_{\varepsilon}(a_{i})}{\int} r^{2}\lambda_{i}\partial_{\lambda_{i}}\varphi_{i}\varphi_{i}^{\frac{n+2}{n-2}}d\mu_{g_{0}} 
= 
\lambda_{i}\partial_{\lambda_{i}}
\int_{\R^{n}} r^{2}\lambda_{i}\partial_{\lambda_{i}}\delta_{0, \lambda_{i}}\delta_{0, \lambda_{i}}^{\frac{n+2}{n-2}} dx \\
= &
\frac{n-2}{2n}(\lambda_{i} \partial_{\lambda_{i}})^{2}
\int r^{2}\delta_{0, \lambda_{i}}^{\frac{2n}{n-2}}dx
=
\frac{n-2}{2n}(\lambda_{i} \partial_{\lambda_{i}})^{2}
(\lambda_{i}^{-2}\int_{\R^{n}}\frac{r^{2}}{(1+r^{2})^{n}}dx) \\
= &
\frac{n-2}{8n}\frac{\bar c_{2}}{\lambda_{i}^{2}}
\end{split}
\end{equation*}
up to some error of order $ o(1) $.
Consequently
\begin{equation*}\begin{split}
\frac{(k_{\tau})_{\alpha^{m}\varphi_{m}+\bar v}^{\frac{2}{p+1}}}{8n(n-1)}
&
\partial^{2}J_{\tau}(\alpha^{m}\varphi_{m}+\bar v)(\lambda_{i}\partial_{\lambda_{i}}\varphi_{i})^{2} 
= 
c_{2}
(
1
+
\frac{n-2}{16n^2}\frac{\tilde c_{1}\bar c_{2}}{\tilde c_{2}c_{2}}
)
\tau
\\
&
 -
\frac{n+2}{n-2}
 \tau 
 \int_{B_{\varepsilon}(a_{i})}\varphi_{i}^{\frac{4}{n-2}}\ln(1+\lambda_{i}^{2}r^{2})^{\frac{n-2}{2}}
\vert\lambda_{i}\partial_{\lambda_{i}}\varphi_{i}\vert^{2}d\mu_{g_{0}} \\
& -
\tau 
\underset{B_{\varepsilon}(a_{i})}{\int}
\ln(1+\lambda_{i}^{2}r^{2})^{\frac{n-2}{2}}
\varphi_{i}^{\frac{n+2}{n-2}}(\lambda_{i}\partial_{\lambda_{i}})^{2}\varphi_{i}
d\mu_{g_{0}}.
\end{split}
\end{equation*}
Finally we calculate passing to integration over $ \R^{n} $ and up to a $ o(1) $ 
\begin{equation*}
\begin{split}
\frac{n+2}{n-2} &
\underset{B_{\varepsilon}(a_{i})}{\int}\varphi_{i}^{\frac{4}{n-2}}\ln(1 +\lambda_{i}^{2}r^{2})^{\frac{n-2}{2}}
\vert\lambda_{i}\partial_{\lambda_{i}}\varphi_{i}\vert^{2}d\mu_{g_{0}} \\
=&
\int_{\R^{n}}
\ln(1+\lambda_{i}^{2}r^{2})^{\frac{n-2}{2}}
\lambda_{i}\partial_{\lambda_{i}}\delta_{0, \lambda_{i}}
\lambda_{i}\partial_{\lambda_{i}}\delta_{0, \lambda_{i}}^{\frac{n+2}{n-2}}dx \\
= & 
\lambda_{i}\partial_{\lambda_{i}}
\int_{\R^{n}}
\ln(1+\lambda_{i}^{2}r^{2})^{\frac{n-2}{2}}
\lambda_{i}\partial_{\lambda_{i}}\delta_{0, \lambda_{i}}
\delta_{0, \lambda_{i}}^{\frac{n+2}{n-2}}dx\\
&-
(n-2)
\int_{\R^{n}}
\frac{\lambda_{i}^{2}r^{2}}{1+\lambda_{i}^{2}r^{2}}
\lambda_{i}\partial_{\lambda_{i}}\delta_{0, \lambda_{i}}
\delta_{0, \lambda_{i}}^{\frac{n+2}{n-2}}dx \\
& -
\int_{\R^{n}}
\ln(1+\lambda_{i}^{2}r^{2})^{\frac{n-2}{2}}
(\lambda_{i}\partial_{\lambda_{i}})^{2}\delta_{0, \lambda_{i}}
\delta_{0, \lambda_{i}}^{\frac{n+2}{n-2}}dx, 
\end{split}
\end{equation*}
where the first summand above vanishes by rescaling, 
and we are reduced to 
\begin{equation*} \begin{split}
\frac{(k_{\tau})_{\alpha^{m}\varphi_{m}+\bar v}^{\frac{2}{p+1}}}{8n(n-1)}
\partial^{2}J_{\tau}(\alpha^{m}\varphi_{m}+\bar v)(\lambda_{i}\partial_{\lambda_{i}}\varphi_{i})^{2} \\
= 
c_{2}
(
1
+
\frac{n-2}{16n^2}\frac{\tilde c_{1}\bar c_{2}}{\tilde c_{2}c_{2}}
)
\tau
+
(n- 2) \tau
\int_{\R^{n}}
\frac{\lambda_{i}^{2}r^{2}}{1+\lambda_{i}^{2}r^{2}}
\lambda_{i}\partial_{\lambda_{i}}\delta_{0, \lambda_{i}}
\delta_{0, \lambda_{i}}^{\frac{n+2}{n-2}}dx
, 
\end{split}
\end{equation*}
where up to some $ o(1) $ and with 
$
\hat c_{3}
=
-
\underset{\R^{n}}{\int}
\frac{r^{2}(1-r^{2})}{(1+r^{2})^{n+2}}dx
$ 
\begin{equation}\label{c_3}
\begin{split}
\underset{\R^{n}}{\int} 
\frac{\lambda_{i}^{2}r^{2}}{1+\lambda_{i}^{2}r^{2}}
\lambda_{i}\partial_{\lambda_{i}}\delta_{0, \lambda_{i}}
\delta_{0, \lambda_{i}}^{\frac{n+2}{n-2}} dx
= 
\frac{n-2}{2}
\underset{\R^{n}}{\int}
\frac{r^{2}(1-r^{2})}{(1+r^{2})^{n+2}}dx
=
-
\frac{n-2}{2}\hat c_{3}
.
\end{split}
\end{equation}
By an explicit computation of the above constants 
we conclude that 
\begin{equation*}\begin{split}
\frac{(k_{\tau})_{\alpha^{m}\varphi_{m}+\bar v}^{\frac{2}{p+1}}}{8n(n-1)}
& \partial^{2}J_{\tau}(\alpha^{m}\varphi_{m}+\bar v)(\lambda_{i}\partial_{\lambda_{i}}\varphi_{i})^{2} \\
= &
\left(
c_{2} 
(
1
+
\frac{n-2}{16n^2}\frac{\tilde c_{1}\bar c_{2}}{\tilde c_{2}c_{2}}
)
-
\frac{(n-2)^{2}}{2}\hat c_{3}
\right)
\tau 
= 
\frac{(n-2)^{2}\Gamma^{2}(\frac{n}{2})}{128 n \Gamma(n+1)}
\tau
\end{split}
\end{equation*}
up to an error $ o(\frac{1}{\lambda^{2}}) $.
Thence with $ i = 1, \dots, q $ and $ k, l = 1, \dots, n $ 
we derive
\begin{equation*} 
[\partial^{2}J_{\tau}(\alpha^{k}\varphi_{k}+\bar v)]_{\mathbb{B}} 
=
 \begin{pmatrix}
\frac{1}{\lambda^2} \mathbb{V}_{+} & 0 & 0 & 0 \\
0 & \frac{1}{\lambda^2} \mathbb{A}_{q-1, 0} & 0 & 0\\
0 &0 & \frac{1}{\lambda^2} \mathbb{\Lambda}_{+} \\
0 & 0 & 0 &
\partial^{2}J_{\tau} \frac{(\nabla_{a_{i}})_{k}}{\lambda_{i}}\varphi_{i} \frac{(\nabla_{a_{i}})_{l}}{\lambda_{i}}\varphi_{i} \\
\end{pmatrix}
\end{equation*}
up to $ o(\frac{1}{\lambda^{2}}) $, where $ \mathbb{\Lambda}_{+}>0 $ is as in the statement.
For instance consider
\begin{equation*}\begin{split}
\frac{(k_{\tau})_{\alpha^{m}\varphi_{m}+\bar v}^{\frac{2}{p+1}}}{8n(n-1)}
&
\partial^{2}J_{\tau}(\alpha^{m}\varphi_{m}+\bar v)
(\frac{(\nabla_{a_{i}})_{_{1}}}{\lambda_{i}}\varphi_{i})^{2}
 \\
= &
(
1+\frac{n+2}{n-2}(\frac{\bar c_{1}}{\bar c_{0}}-\frac{\tilde c_{1}}{\tilde c_{2}}\frac{\bar c_{2}}{ \bar c_{0}})
)
\tau
\int \varphi_{i}^{\frac{4}{n-2}}\vert \frac{(\nabla_{a_{i}})_{_{1}}}{\lambda_{i}}\varphi_{i}\vert^{2} d\mu_{g_{0}} \\\
& -
\frac{n+2}{n-2}
 \tau 
\underset{B_{\varepsilon}(a_{i})}{\int}\varphi_{i}^{\frac{4}{n-2}}\ln(1+\lambda_{i}^{2}r^{2})^{\frac{n-2}{2}}
\vert \frac{(\nabla_{a_{i}})_{_{1}}}{\lambda_{i}}\varphi_{i}\vert^{2}d\mu_{g_{0}}\\
&
 -
\tau 
\underset{B_{\varepsilon}(a_{i})}{\int}
\biggr(
\ln(1+\lambda_{i}^{2}r^{2})^{\frac{n-2}{2}}
-
\frac{\bar c_{1}}{c_{1}}
\biggr)
\varphi_{i}^{\frac{n+2}{n-2}} 
(\frac{(\nabla_{a_{i}})_{_{1}}}{\lambda_{i}})^{2}\varphi_{i}
d\mu_{g_{0}} \\
& +
\tau 
\underset{B_{\varepsilon}(a_{i})}{\int}
\biggr(
\frac{\tilde c_{1}}{\tilde c_{2}}\frac{ \lambda_{i}^{2}r^{2}}{2n}
 -
\frac{\tilde c_{1}\bar c_{2}}{\tilde c_{2}c_{1}}
\biggr) \varphi_{i}^{\frac{n+2}{n-2}}(\frac{(\nabla_{a_{i}})_{_{1}}}{\lambda_{i}})^{2}\varphi_{i} d\mu_{g_{0}}
\\
& 
-
\underset{B_{\varepsilon}(a_{i})}{\int}
(
\frac{\nabla^{2}K_{i}}{2K_{i}}
x^{2}
-
\frac{\Delta K_{i}}{2nK_{i}}r^{2}
)
\varphi_{i}^{\frac{n+2}{n-2}}(\frac{(\nabla_{a_{i}})_{_{1}}}{\lambda_{i}})^{2}\varphi_{i} d\mu_{g_{0}} \\
& -
\frac{n+2}{n-2}
\underset{B_{\varepsilon}(a_{i})}{\int} \frac{\nabla^{2}K_{i}}{2K_{i}}x^{2}\varphi_{i}^{\frac{4}{n-2}}\vert \frac{(\nabla_{a_{i}})_{_{1}}}{\lambda_{i}}\varphi_{i}\vert^{2}d\mu_{g_{0}}.
\end{split}
\end{equation*}
At this point some simplifications occur. From the relation 
\begin{equation*}
\tilde c_{1}\tau + \tilde c_{2}\frac{\Delta K_{i}}{K_{i}\lambda_{i}^{2}}=o(\frac{1}{\lambda^{2}})
\end{equation*} 
we obtain cancellation of the terms involving $ \Delta K_i $ and $ \frac{\tilde c_{1}}{\tilde c_{2}}\frac{ \lambda_{i}^{2}r^{2}}{2n} $. 
Using 
 $$ 
\int \varphi_{i}^{\frac{4}{n-2}}\vert \frac{(\nabla_{a_{i}})_{_{1}}}{\lambda_{i}}\varphi_{i}\vert^{2}d\mu_{g_{0}}=\frac{c_{3}}{n} + o(1)
$$
and 
$$
\quad\int \varphi_{i}^{\frac{n+2}{n-2}}
(\frac{(\nabla_{a_{i}})_{_{1}}}{\lambda_{i}})^{2}\varphi_{i}
d\mu_{g_{0}}
=
-\frac{n+2}{n-2}\frac{c_{3}}{n}
+
o(1)
 $$ 
as well as $ (\frac{\bar c_{1}}{\bar c_{0}}-\frac{\tilde c_{1}}{\tilde c_{2}}\frac{\bar c_{2}}{ \bar c_{0}}) \frac{c_3}{n}
=(\frac{\bar c_{1}}{c_{1}}-\frac{\tilde c_{1}}{\tilde c_{2}}\frac{\bar c_{2}}{c_{1}})\frac{c_2}{n} $ due $ \bar c_{0}=c_{1} $ 
and $ c_2 = c_3 $, we find
\begin{equation*}\begin{split}
\frac{(k_{\tau})_{\alpha^{m}\varphi_{m}+\bar v}^{\frac{2}{p+1}}}{8n(n-1)}
&
\partial^{2}J_{\tau}(\alpha^{m}\varphi_{m}+\bar v)
(\frac{(\nabla_{a_{i}})_{_{1}}}{\lambda_{i}}\varphi_{i})^{2}
 \\
= &
\frac{c_{3}}{n}
\tau
 -
\frac{n+2}{n-2}
 \tau 
\underset{B_{\varepsilon}(a_{i})}{\int}\varphi_{i}^{\frac{4}{n-2}}\ln(1+\lambda_{i}^{2}r^{2})^{\frac{n-2}{2}}
\vert \frac{(\nabla_{a_{i}})_{_{1}}}{\lambda_{i}}\varphi_{i}\vert^{2}d\mu_{g_{0}} \\
& -
\tau 
\underset{B_{\varepsilon}(a_{i})}{\int}
\ln(1+\lambda_{i}^{2}r^{2})^{\frac{n-2}{2}}
\varphi_{i}^{\frac{n+2}{n-2}} 
(\frac{(\nabla_{a_{i}})_{_{1}}}{\lambda_{i}})^{2}\varphi_{i}
d\mu_{g_{0}} 
\\
& 
-
\underset{B_{\varepsilon}(a_{i})}{\int}
\frac{\nabla^{2}K_{i}}{2K_{i}}
x^{2}
\varphi_{i}^{\frac{n+2}{n-2}}(\frac{(\nabla_{a_{i}})_{_{1}}}{\lambda_{i}})^{2}\varphi_{i} d\mu_{g_{0}} \\
& -
\frac{n+2}{n-2}
\underset{B_{\varepsilon}(a_{i})}{\int} \frac{\nabla^{2}K_{i}}{2K_{i}}x^{2}\varphi_{i}^{\frac{4}{n-2}}\vert \frac{(\nabla_{a_{i}})_{_{1}}}{\lambda_{i}}\varphi_{i}\vert^{2}d\mu_{g_{0}}.
\end{split}
\end{equation*}
Moreover we have, passing to integration over $ \R^{n} $, up to an error $ o(1) $ 
\begin{equation*}
\begin{split}
\frac{n+2}{n-2} &
\underset{B_{\varepsilon}(a_{i})}{\int}\varphi_{i}^{\frac{4}{n-2}}\ln(1+\lambda_{i}^{2}r^{2})^{\frac{n-2}{2}}
\vert \frac{(\nabla_{a_{i}})_{_{1}}}{\lambda_{i}}\varphi_{i}\vert^{2}d\mu_{g_{0}} \\
= &
\int_{\R^{n}}\ln(1+\lambda_{i}^{2}r^{2})^{\frac{n-2}{2}}
\frac{(\nabla_{a_{i}})_{_{1}}}{\lambda_{i}}\delta_{0, \lambda_{i}}^{\frac{n+2}{n-2}}
\frac{(\nabla_{a_{i}})_{_{1}}}{\lambda_{i}}\delta_{0, \lambda_{i}} dx \\
= &
-(n-2)
\int_{\R^{n}}\frac{\lambda_{i}x_{1}}{1+\lambda_{i}^{2}r^{2}}
\delta_{0, \lambda_{i}}^{\frac{n+2}{n-2}} 
\frac{(\nabla_{a_{i}})_{_{1}}}{\lambda_{i}}\delta_{0, \lambda_{i}} dx \\
& -
\int_{\R^{n}}\ln(1+\lambda_{i}^{2}r^{2})^{\frac{n-2}{2}}
\delta_{0, \lambda_{i}}^{\frac{n+2}{n-2}}
(\frac{(\nabla_{a_{i}})_{_{1}}}{\lambda_{i}})\delta_{0, \lambda_{i}} dx \\
\end{split} 
\end{equation*}
and find for the first summand 
\begin{equation*}
\begin{split}
(n-2)
\int_{\R^{n}}\frac{\lambda_{i}x_{1}}{1+\lambda_{i}^{2}r^{2}}
\delta_{0, \lambda_{i}}^{\frac{n+2}{n-2}}
\frac{(\nabla_{a_{i}})_{_{1}}}{\lambda_{i}}\delta_{0, \lambda_{i}}dx
= & 
-
\int_{\R^{n}}\delta_{0, \lambda_{i}}^{\frac{4}{n-2}}
\vert\frac{(\nabla_{a_{i}})_{_{1}}}{\lambda_{i}}\delta_{0, \lambda_{i}}\vert^{2}
dx =
-\frac{c_{3}}{n}.
\end{split}
\end{equation*}
We therefore are left with 
\begin{equation*}\begin{split}
\frac{(k_{\tau})_{\alpha^{m}\varphi_{m}+\bar v}^{\frac{2}{p+1}}}{8n(n-1)}
&
\partial^{2}J_{\tau}(\alpha^{m}\varphi_{m}+\bar v)
(\frac{(\nabla_{a_{i}})_{_{1}}}{\lambda_{i}}\varphi_{i})^{2}
 \\
= &
-
\underset{B_{\varepsilon}(a_{i})}{\int}
\frac{\nabla^{2}K_{i}}{2K_{i}}
x^{2}
\varphi_{i}^{\frac{n+2}{n-2}}(\frac{(\nabla_{a_{i}})_{_{1}}}{\lambda_{i}})^{2}\varphi_{i} d\mu_{g_{0}} \\
& -
\frac{n+2}{n-2}
\underset{B_{\varepsilon}(a_{i})}{\int} \frac{\nabla^{2}K_{i}}{2K_{i}}x^{2}\varphi_{i}^{\frac{4}{n-2}}\vert \frac{(\nabla_{a_{i}})_{_{1}}}{\lambda_{i}}\varphi_{i}\vert^{2}d\mu_{g_{0}}.
\end{split}
\end{equation*}
Finally passing to integration over $ \R^{n} $ up to some $ o(1) $ 
 there holds
\begin{equation*}
\begin{split}
\frac{n+2}{n-2} &
\underset{B_{\varepsilon}(a_{i})}{\int} x_{l}^{2}\varphi_{i}^{\frac{4}{n-2}}\vert \frac{(\nabla_{a_{i}})_{_{1}}}{\lambda_{i}}\varphi_{i}\vert^{2}d\mu_{g_{0}} 
= 
\int_{\R^{n}}x_{l}^{2} \frac{(\nabla_{a_{i}})_{_{1}}}{\lambda_{i}}\delta_{0, \lambda_{i}} ^{\frac{n+2}{n-2}} \frac{(\nabla_{a_{i}})_{_{1}}}{\lambda_{i}}\delta_{0, \lambda_{i}} dx \\
= &
- 2 \delta_{1, l}
\int_{\R^{n}}\frac{x_{1}}{\lambda_{i}} \delta_{0, \lambda_{i}} ^{\frac{n+2}{n-2}} \frac{(\nabla_{a_{i}})_{_{1}}}{\lambda_{i}}\delta_{0, \lambda_{i}} dx
-
\int_{\R^{n}}x_{l}^{2} \frac{(\nabla_{a_{i}})_{_{1}}}{\lambda_{i}}\delta_{0, \lambda_{i}} ^{\frac{n+2}{n-2}} (\frac{(\nabla_{a_{i}})_{_{1}}}{\lambda_{i}})^{2}\delta_{0, \lambda_{i}} dx, 
\end{split}
\end{equation*}
and similarly for $ j = 2, \dots, n $. Diagonalizing the Hessian we have 
$$ \nabla^{2}K_{i}x^{2}=\sum_{l=1}^{n}\partial_{l}^{2}K_{i}x_{l}^{2} $$ and 
\begin{equation*}
\begin{split}
\int_{\R^{n}}\frac{x_{1}}{\lambda_{i}} \delta_{0, \lambda_{i}} ^{\frac{n+2}{n-2}} \frac{(\nabla_{a_{i}})_{_{1}}}{\lambda_{i}}\delta_{0, \lambda_{i}} dx
= &
-(n-2)\int_{\R^{n}}\delta_{0, \lambda_{i}}^{\frac{2n}{n-2}}\frac{x_{1}^{2}}{1+\lambda_{i}^{2}r^{2}} dx \\
= &
-\frac{n-2}{n\lambda_{i}^{2}}
\int_{\R^{n}}\frac{r^{2}}{(1+r^{2})^{n+1}} dx, 
\end{split}
\end{equation*}
so we conclude that 
\begin{equation*}\begin{split}
\frac{(k_{\tau})_{\alpha^{m}\varphi_{m}+\bar v}^{\frac{2}{p+1}}}{8n(n-1)}
&
\partial^{2}J_{\tau}(\alpha^{m}\varphi_{m}+\bar v)
(\frac{(\nabla_{a_{i}})_{_{1}}}{\lambda_{i}}\varphi_{i})^{2}
= 
-
c\frac{\partial_{1}^{2}K_{i}}{K_{i}\lambda_{i}^{2}}.
\end{split}
\end{equation*}
Similarly one can show analogous formula for any couple of indices 
 $$ 
 \frac{(k_{\tau})_{\alpha^{m}\varphi_{m}+\bar v}^{\frac{2}{p+1}}}{8n(n-1)}
 \partial^{2}J_{\tau}(\alpha^{m}\varphi_{m}+\bar v)
 \frac{(\nabla_{a_{i}})_{_{k}}}{\lambda_{i}}\varphi_{i} \frac{(\nabla_{a_{i}})_{_{l}}}{\lambda_{i}}\varphi_{i} = 
 -
 c\frac{\partial_{k, l}^{2}K_{i}}{K_{i}\lambda_{i}^{2}}.
 $$ 
The proof is thereby complete. 
\end{proof}

\medskip

\noindent
From Proposition \ref{lem_refined_second_variation} we deduce that the kernel of $ \partial^2 J_\tau $ 
is potentially one-dimensional. On the other hand the presence of an at one dimensional kernel at a solutionis necessary
due to the scaling invariance of $ J_\tau $. Hence it is natural to impose somehomogeneous constraint. 

\begin{corollary}
\label{cor_restricted_second_variation}   
Let 
$ u \in \bar{V}(q, \varepsilon)$ 
be a solution of \eqref{eq:scin-tau} 
and
$$ 
I_\tau=J_\tau\lfloor_{[\Vert \cdot\Vert_{L_{g_{0}}}=1]}
\; \text{ or }\;
I_\tau=J_\tau\lfloor_{[\Vert \cdot\Vert_{k_{\tau}}=1]}.
 $$ 
Then, if $ \tilde{u} $ denoted the corresponding normalization of $u$, we have
\begin{equation*}
m(I_\tau, \tilde{u})
=
q-1+\sum_{i=1}^q (n-m(K, a_{i})).
\end{equation*}
\end{corollary}

\section{Appendix: Some estimates and list of constants}\label{s:app}

In this appendix, recalling our notation, we collect some useful statements and formulae
proved in \cite{MM1}.

\begin{lemma}\label{lem_emergence_of_the_regular_part}
There holds 
 $ 
L_{g_{0}}\varphi_{a, \lambda}=
O
(
\varphi_{a, \lambda}^{\frac{n+2}{n-2}}
).
 $ 
More precisely on a geodesic ball $ B_{\alpha}( a ) $ for $ \alpha>0 $ small 
\begin{equation*}
\begin{split}
L_{g_{0}}\varphi_{a, \lambda}
= &
4n(n-1)\varphi_{a, \lambda}^{\frac{n+2}{n-2}}
-
2nc_{n}
r_{a}^{n-2}((n-1)H_{a}+r_{a}\partial_{r_{a}}H_{a}) \varphi_{a, \lambda}^{\frac{n+2}{n-2}} 
\\& +
\frac{R_{g_{a}}}{\lambda} u_{a}^{\frac{2}{n-2}} \varphi_{a, \lambda}^{\frac{n}{n-2}}
+
o(r_{a}^{n-2})\varphi_{a, \lambda}^{\frac{n+2}{n-2}}, \quad
r_{a}=d_{g_{a}}(a, \cdot).
\end{split}
\end{equation*}
Since in conformal normal coordinates $$ R_{g_{a}}=O(r_{a}^{2}),$$ 
cf. \cite{lp}, we obtain
\begin{enumerate}[label=(\roman*)]
 \item \quad for $ n=5 $ 
 $$ 
L_{g_{0}}\varphi_{a, \lambda}
= 
4n(n-1)
[1
-
\frac{c_{n}}{2}r_{a}^{n-2}(
H_{a}(a) + n \nabla H_{a}(a)x
)
]
\varphi_{a, \lambda}^{\frac{n+2}{n-2}} 
 +
O(\lambda^{-2}\varphi_{a, \lambda});
 $$ 
\item \quad for $ n=6 $ and with $ W(a) = |\mathbb{W}(a)|^2 $ 
 $$ L_{g_{0}}\varphi_{a, \lambda}=
4n(n-1)\varphi_{a, \lambda}^{\frac{n+2}{n-2}}
=
4n(n-1)[1+\frac{c_{n}}{2}W(a)\ln r]\varphi_{a, \lambda}^{\frac{n+2}{n-2}}
+
O
(\lambda^{-2}\varphi_{a, \lambda}); $$ 
\item \quad for $ n=7 $ 
 $$ L_{g_{0}}\varphi_{a, \lambda}=
4n(n-1)\varphi_{a, \lambda}^{\frac{n+2}{n-2}}
+
O
(\lambda^{-2}\varphi_{a, \lambda}). $$ 
\end{enumerate}
These expansions persist upon taking $ \lambda \partial \lambda $ and $ \frac{\nabla_{a}}{\lambda} $ derivatives.
\end{lemma}

\medskip

\begin{lemma}\label{lem_interactions}   
Let $ \theta=\frac{n-2}{2}\tau $ and $ k, l=1, 2, 3 $ and $ i, j = 1, \ldots, q $. 
There holds uniformly as $ 0\leq \tau\longrightarrow 0 $ 
\begin{enumerate}[label=(\roman*)]
 \item \quad
 $ \vert \phi_{k, i}\vert, 
\vert \lambda_{i}\partial_{\lambda_{i}}\phi_{k, i}\vert, 
\vert \frac{1}{\lambda_{i}}\nabla_{a_{i}} \phi_{k, i}\vert
\leq 
C \varphi_{i}; $ 
 \item \quad
 $ 
\lambda_{i}^{\theta}\int \varphi_{i}^{\frac{4}{n-2}-\tau} \phi_{k, i}\phi_{k, i}d\mu_{g_{0}}
=
c_{k}\cdot id
+
O(\tau +\frac{1}{\lambda_{i}^{2+\theta}}), \;c_{k}>0; $ 
\item \quad 
for $ i\neq j $ up to some error of order 
 $ O(\tau^{2}+\sum_{i\neq j}(\frac{1}{\lambda_{i}^{4}}+\varepsilon_{i, j}^{\frac{n+2}{n}})) $ 
\begin{equation*}
\lambda_{i}^{\theta}\int \varphi_{i}^{\frac{n+2}{n-2}-\tau}\phi_{k, j}d\mu_{g_{0}}
= 
b_{k}d_{k, i}\varepsilon_{i, j}
=
\int \varphi_{i}^{1-\tau}d_{k, j}\varphi_{j}^{\frac{n+2}{n-2}}d\mu_{g_{0}} ;\end{equation*}
 \item \quad 
 $ 
\lambda_{i}^{\theta}\int \varphi_{i}^{\frac{4}{n-2}-\tau} \phi_{k, i}\phi_{l, i}d\mu_{g_{0}}
= 
O(\frac{1}{\lambda_{i}^{2}}) $ 
for $ k\neq l $ and for $ k=2, 3 $ 
 $$ \textstyle\quad 
\lambda_{i}^{\theta}\int \varphi_{i}^{\frac{n+2}{n-2}-\tau} \phi_{k, i}d\mu_{g_{0}}
=
O\left(
\tau
+
\begin{pmatrix}
\lambda_{i}^{2-n} & \text{for } n = 5 \\
\frac{\ln \lambda_{i}}{\lambda_{i}^{4}} & \text{for } n=6 \\
\lambda_{i}^{4} &\text{for } n\geq 7
\end{pmatrix}
\right);
 $$ 
 \item \quad
for $ i\neq j, \;\alpha +\beta=\frac{2n}{n-2}, \; \alpha-\tau>\frac{n}{n-2}>\beta\geq 1 $ 
 $$ 
\lambda_{i}^{\theta}\int \varphi_{i}^{\alpha-\tau }\varphi_{j}^{\beta} d\mu_{g_{0}}
=
O(\varepsilon_{i, j}^{\beta})
; $$ 
\item \quad
 $ 
\int \varphi_{i}^{\frac{n}{n-2}}\varphi_{j}^{\frac{n}{n-2}} d\mu_{g_{0}}
=
O(\varepsilon^{\frac{n}{n-2}}_{i, j}\ln \varepsilon_{i, j}), \, i\neq j;
 $ 
 \item \quad 
 $ 
(1, \lambda_{i}\partial_{\lambda_{i}}, \frac{1}{\lambda_{i}}\nabla_{a_{i}})\varepsilon_{i, j}=O(\varepsilon_{i, j})
, \, i\neq j $, 
\end{enumerate}
 cf. \eqref{eq:eijm}, 
with constants 
\begin{itemize}
\item \quad
$
b_{k}=\underset{\R^{n}}{\int}\frac{dx}{(1+r^{2})^{\frac{n+2}{2}}}
$
for $k=1,2,3$;
\item \quad
$
c_{1}=\underset{\R^{n}}{\int}\frac{dx}{(1+r^{2})^{n}};
$
\item \quad
$
c_{2}=\frac{(n-2)^{2}}{4}\underset{\R^{n}}{\int}\frac{\vert r^{2}-1\vert^{2}dx}{(1+r^{2})^{n+2}};
$
\item \quad
$
c_{3}=\frac{(n-2)^{2}}{n}\underset{\R^{n}}{\int}\frac{r^{2}dx}{(1+r^{2})^{n+2}}.
$
\end{itemize}
\end{lemma}

\begin{lemma}\label{lem_testing_with_v}
For $u\in V(q, \varepsilon)$ with $k_{\tau}=1$ and 
$\nu\in H_{u}(q, \varepsilon)$ there holds
\begin{equation*}
\partial J_{\tau}(\alpha^{i}\varphi_{i})\nu
= 
O\bigg(
\bigg[
\sum_{r}\frac{\tau}{\lambda_{r}^{\theta}}
+
\sum_{r}\frac{\vert \nabla K_{r}\vert}{\lambda_{r}^{1+\theta}}
+
\sum_{r}\frac{1}{\lambda_{r}^{2+\theta}}
+
\sum_{r\neq s}\frac{\varepsilon_{r, s}^{\frac{n+2}{2n}}}{\lambda_{r}^{\theta}}
\bigg]
\Vert \nu \Vert \bigg).
\end{equation*}
\end{lemma}

\begin{lemma}
 \label{lem_alpha_derivatives_at_infinity}
 For $ u\in V(q, \varepsilon) $ and $ \varepsilon>0 $ sufficiently small the three quantities 
 $ \partial J_{\tau} (u)\phi_{1, j} $, 
 $ \partial J_{\tau}(\alpha^{i}\varphi_{i})\phi_{1, j} $, 
 $ \partial_{\alpha_{j}}J_{\tau}(\alpha^{i}\varphi_{i}) $ can be written as 
 \begin{equation*}
 \begin{split}
 \frac{\alpha_{j}}
 {(\alpha_{K, \tau}^{\frac{2n}{n-2}})^{\frac{n-2}{n}}}
 \bigg( 
 &
 \grave c_{0}\big(
 1
 -
 \frac{\alpha^{2}}{\alpha_{K, \tau}^{p+1}}
 \frac{K_{j}}{\lambda_{j}^{\theta}}\alpha_{j}^{p-1}
\big)
 -
 \grave c_{2}
 \big(
 \frac{\Delta K_{j}}{K_{j}\lambda_{j}^{2}}
 -
 \sum_{k}\frac{\Delta K_{k}}{K_{k}\lambda_{k}^{2}}
 \frac{\alpha_{k}^{2}}{\alpha^{2}}
 \big)
 \\ & 
 +
 \grave b_{1} \bigg(
 \sum_{k\neq l}
 \frac{\alpha_{k}\alpha_{l}}{\alpha^{2}}
 \varepsilon_{k, l}
 -
 \sum_{j\neq i}\frac{\alpha_{i}}{\alpha_{j}}\varepsilon_{i, j}
 \bigg)
 \\ & \quad \quad\quad\quad\quad \;\, 
 -
 \grave d_{1}
 \begin{pmatrix}
 \frac{H_{j}}{\lambda_{j}^{3}} -\sum_{k}\frac{\alpha_{k}^{2}}{\alpha^{2}}\frac{H_{k}}{\lambda_{k}^{3 }} &\text{for } n=5\\
 \frac{W_{j}\ln \lambda_{j}}{\lambda_{i}^{4}}-\sum_{k}\frac{\alpha_{k}^{2}}{\alpha^{2}}\frac{W_{k}\ln \lambda_{k}}{\lambda_{k}^{4}} &\text{for } n=6\\
 0 &\text{for }n\geq 7
 \end{pmatrix} 
 \bigg)
 \end{split}
 \end{equation*}
up to an error of order 
 $$ 
O
\big(
\tau^{2}
+
\sum_{r\neq s} 
\frac{\vert \nabla K_{r}\vert^{2}}{\lambda_{r}^{2}}
+
\frac{1}{\lambda_{r}^{4}}
+
\varepsilon_{r, s}^{\frac{n+2}{n}}
+
\vert \partial J_{\tau}(u)\vert^{2}
\big), 
 $$ 
with positive constants 
\begin{equation}\label{constants_alpha_derivative}
\begin{split}
& \bullet \quad  
 \grave b_{1}
 =
 \frac{8n(n-1)(n+2)}{\bar c_{0}^{\frac{n-2}{n}}(n-2)}b_{1}; \\
& \bullet \quad
 \grave c_{2}
 =
 \frac{{8n}(n-1)}{\bar c_{0}^{\frac{n-2}{n}}}
 \bar c_{2};\\
& \bullet \quad
 \grave d_{1}
 =
 \frac{{8n}(n-1)}{\bar c_{0}^{\frac{n-2}{n}}}
 \bar d_{1};\\
& \bullet \quad
 \grave c_{0}= 8n(n-1)\bar c_{0}^{\frac{2}{n}}.
\end{split}
 \end{equation}
 In particular for all $ j $ 
 \begin{equation*}
 \frac{\alpha^{2}}{\alpha_{K, \tau}^{p+1}}\frac{K_{j}}{\lambda_{j}^{\theta}}\alpha_{j}^{p-1}
 =
 1
 +
 O 
 \big(
 \tau
 +
 \sum_{r\neq s} 
 \frac{1}{\lambda_{r}^{2}}
 +
 \varepsilon_{r, s}
 +
 \vert \partial J_{\tau}(u)\vert
 \big).
\end{equation*}
\end{lemma}

\begin{lemma}
 \label{lem_lambda_derivatives_at_infinity} 
 For $ u\in V(q, \varepsilon) $ and $ \varepsilon>0 $ sufficiently small the three quantities 
 $ \partial J_{\tau}(u) \phi_{2, j} $, 
 $ \partial J_{\tau}(\alpha^{i}\varphi_{i})\phi_{2, j} $ and 
 $ \frac{\lambda_{j}}{\alpha_{j}}\partial_{\lambda_{j}}J_{\tau}(\alpha^{i}\varphi_{i}) $ 
 can be written as 
 \begin{equation*}
 \begin{split}
 \frac{\alpha_{j}}{(\alpha_{K, \tau}^{\frac{2n}{n-2}})^{\frac{n-2}{n}}}
 \bigg(
 \tilde c_{1}\tau
 +
 \tilde c_{2}\frac{\Delta K_{j}}{K_{j}\lambda_{j}^{2}} 
 -
 \tilde b_{2}\sum_{j\neq i }\frac{\alpha_{i}}{\alpha_{j}}\lambda_{j}\partial_{\lambda_{j}}\varepsilon_{i, j} 
+
 \tilde d_{1}
 \begin{pmatrix}
 \frac{H_{j}}{\lambda_{j}^{3 }} &\text{for }\; n=5\\
 \frac{W_{j}\ln \lambda_{j}}{\lambda_{j}^{4}} &\text{for }\; n=6\\
 0 &\text{for }\; n\geq 7
 \end{pmatrix} 
 \bigg), 
 \end{split}
 \end{equation*}
 with positive constants $ \tilde c_{1}, \tilde c_{2}, \tilde d_{1}, \tilde b_{2} $ up to some error 
 $$ 
 O
 \big(
 \tau^{2}
 +
 \sum_{r\neq s}
 \frac{\vert \nabla K_{r}\vert^{2}}{\lambda_{r}^{2}}
 +
 \frac{1}{\lambda_{r}^{4}}
 +
 \varepsilon_{r, s}^{\frac{n+2}{n}}
 +
 \vert \partial J_{\tau}(u)\vert^{2}
 \big).
 $$ \end{lemma}

\begin{lemma} 
\label{lem_a_derivatives_at_infinity} 
For $ u\in V(q, \varepsilon) $ and $ \varepsilon>0 $ sufficiently small the three quantities 
 $ \partial J_{\tau}(u)\phi_{3, j} $, 
 $ \partial J_{\tau}(\alpha^{i}\varphi_{i})\phi_{3, j} $ and $ \frac{\nabla_{a_{j}}}{\alpha_{j}\lambda_{j}}J_{\tau}(\alpha^{i}\varphi_{i}) $ can be written as 
\begin{equation*}
\begin{split}
-\frac{\alpha_{j}}{(\alpha_{K, \tau}^{\frac{2n}{n-2}})^{\frac{n-2}{n}}}
\left(
\check{c}_{3}\frac{\nabla K_{j}}{K_{j}\lambda_{j}}
+
\check{c}_{4} \frac{\nabla \Delta K_j}{K_j \lambda_{j}^3}
+
\check{b}_{3}
\sum_{j\neq i}
\frac{\alpha_{i}}{\alpha_{j}}
\frac{\nabla_{a_{j}}}{\lambda_{j}}\varepsilon_{i, j}
\right), 
\end{split}
\end{equation*}
with positive constants $ \check{c}_{3}, \check{c}_{4}, \check{b}_{3} $ up to some error 
 $$ 
O\big(
\tau^{2}
+
\sum_{r\neq s}
\frac{\vert \nabla K_{r}\vert^{2}}{\lambda_{r}^{2}}
+
\frac{1}{\lambda_{r}^{4}}
+
\varepsilon_{r, s}^{\frac{n+2}{n}}
+
\vert \partial J_{\tau}(u)\vert^{2}
\big).
 $$ 
\end{lemma}

 \begin{lemma} 
\label{lem_upper_bound}
For every $ u\in V(q, \varepsilon) $ there holds
\begin{equation*}
\vert \partial J_\tau(u)\vert \
\lesssim 
\tau
+
\sum_{r\neq s}\frac{\vert \nabla K_{r}\vert}{\lambda_{r}}
+
\frac{1}{\lambda_{r}^{2}}
+
\vert 1
-
\frac{\alpha^{2}}{\alpha_{K, \tau}^{p+1}}
\frac{K_{r}}{\lambda_{r}^{\theta}}\alpha_{r}^{p-1} \vert
+
\varepsilon_{r, s}^{\frac{n+2}{2n}}
+
\Vert v \Vert. 
\end{equation*}
\end{lemma}

\begin{thm}
\label{lem_top_down_cascade} 
Suppose that $ n \geq 5 $, $ K : M \longrightarrow \R_{+} $ is Morse and satisfies \eqref{eq:nd}. 
Then for $ \varepsilon>0 $ sufficiently small there exists $ c>0 $ such that for any 
\begin{equation*}
u\in V(q, \varepsilon) \quad \text{ with } \quad  k_{\tau}=1
\end{equation*}
there holds 
\begin{equation*}
\vert \partial J(u)\vert
\geq
c\big( 
\tau +\sum_{r\neq s}\frac{\vert \nabla K_{r}\vert}{\lambda_{r}}+\frac{1}{\lambda_{r}^{2}}
+
\big\vert 1
-
\frac{\alpha^{2}}{\alpha_{K, \tau}^{p+1}}
\frac{K_{r}}{\lambda_{r}^{\theta}}\alpha_{r}^{p-1} \big\vert
+\varepsilon_{r, s}\big), 
\end{equation*}
unless 
there is a violation of at least one of the four conditions 
\begin{enumerate}[label=(\roman*)]
 \item \quad 
 $ \tau>0 $ ; 
 \item \quad 
 $ 
\exists \; x_{i}\neq x_{j}
\in 
\{\nabla K=0\}\cap \{\Delta K<0\}
 $ 
and
 $ d(a_{i}, x_{i})=O(\frac{1}{\lambda_{i}}) $ ;
 \item \quad 
 $ \alpha_{j}=\Theta\, \cdot(\frac{\lambda_{j}^{\theta}}{K_{j}})^{\frac{1}{p-1}}+o(\frac{1}{\lambda_{j}^{2}}); $ 
 \item \quad 
 $ \tilde c_{1}\tau
=
-
\tilde c_{2}\frac{\Delta K_{k}}{K_{k}\lambda_{k}^{2}} 
+
o(\frac{1}{\lambda_{k}^{2}}) $ 
\end{enumerate}
where 
 $ \Theta $ is a positive constant, uniformly bounded and bounded away from zero, that depends on $ u $, 
cf. Remark 6.2 in \cite{MM1}. In the latter case there holds 
 $$ \lambda_{1}\simeq \ldots \lambda_{q}\simeq \lambda=\frac{1}{\sqrt{\tau}} $$ 
and setting $ a_{j}=\exp_{g_{x_{j}}}(\bar a_{j}) $, 
we still have up to an error $ o(\frac{1}{\lambda^{3}}) $ the lower bound
\begin{equation*}
\begin{split}
\vert \partial J(u) \vert 
\gtrsim &
\sum_{j}
\vert
\tau
 +
\frac{2}{9}\frac{\Delta K(x_{j})}{K(x_{j})\lambda_{j}^{2}} 
+
\frac{512}{9\pi}
[
\frac{H(x_{j})}{\lambda_{j}^{3}}
+
\sum_{j\neq i }\sqrt{\frac{K(x_{j})}{K(x_{i})}} \frac{G_{g_{0}}(x_{i}, x_{j})}{\gamma_{n}(\lambda_{i}\lambda_{j})^{\frac{3}{2}}}
] 
\vert
 \\
& 
+
\sum_{j}
\vert 
\frac{\bar a_{j}}{\lambda_{j}}
+
\frac{\check c_{4}}{\check c_{3}}
(\nabla^{2}K(x_{j}))^{-1}
\frac{\nabla \Delta K(x_{j})}{\lambda_{j}^3} \vert 
\\
& +
\sum_{j}
\vert 
\alpha_{j}
-
\Theta \cdot
\sqrt[p-1]
{
\frac{\lambda_{j}^{\theta}}{K(a_{j})}
(
1
-
\frac{1}{90}
\left(
\frac{\Delta K(x_{j})}{K(x_{j})\lambda_{j}^{2}}
+
\frac{2816}{\pi}
\frac{H(x_{j})}{\lambda_{j}^{3}}
-
\frac
{
\sum_{k}
(\frac{\Delta K(x_{k})}{K(x_{k})^{2}\lambda_{k}^{2}}
+
\frac{2816}{\pi}
\frac{H(x_{k})}{K(x_{k})\lambda_{k}^{3}}
)
}
{
\sum_{k}\frac{1}{K(x_{k})}
}
\right)
)
}
\vert
 \end{split}
\end{equation*}
in case $ n=5 $ and 
\begin{equation*}
\begin{split}
\vert \partial J(u)\vert
\gtrsim &
\sum_{j}
(
\vert \tau+\frac{\tilde c_{2}}{\tilde c_{1}}\frac{\Delta K(x_{j})}{K(x_{j})\lambda_{j}^{2}}\vert \\
& +
\vert
\frac{ \bar a_{j}}{\lambda_{j}}
 +
\frac{\check{c}_{4}}{\check{c}_{3}} (\nabla^{2}K(x_{j}))^{-1}\frac{\nabla \Delta K(x_{j})}{\lambda_{j}^3}
\vert
+
\vert 
\alpha_{j}
-
\Theta \cdot
\sqrt[p-1]
{
\frac{\lambda_{j}^{\theta}}{K(a_{j})}}
\vert
)
\end{split}
\end{equation*}
in case $ n\geq 6 $. 
The constants appearing above are defined by
\begin{itemize}
\item \quad
 $ \bar c_{0}=\int_{\R^{n}}\frac{dx}{(1+r^{2})^n} $ ;
\item \quad
 $ \tilde c_{1}=\frac{n(n-1)(n-2)^{2}}{\bar c_{0}^{\frac{n-2}{n}}}\underset{\R^{n}}{\int}\frac{1-r^{2}}{(1+r^{2})^{n+1}}\ln\frac{1}{1+r^{2}}dx; $ 
\item \quad
 $ \tilde c_{2}=-\frac{(n-1)(n-2)}{\bar c_{0}^{\frac{n-2}{n}}}\underset{\R^{n}}{\int}\frac{r^{2}(1-r^{2})}{(1+r^{2})^{n+1}}dx; $ 
\item \quad
 $ \check c_{3}=\int_{\R^{n}}\frac{4(n-1)(n-2)}{(1+r^{2})^{n}}dx; $ 
\item \quad
 $ \tilde b_{2} = \frac{4n(n-1)}{\bar c_{0}^{\frac{n-2}{n}}} \int_{\R^n} \frac{dx}{(1+r^2)^{\frac{n+2}{2}}} $ ; 
\item \quad
 $ \check c_{4}=\int_{\R^{n}}\frac{2(n-1)r^{2}}{(1+r^{2})^{n}}dx $ ; 
\item \quad
 $ \tilde d_{1} = \frac{4n(n-1)}{\bar c_{0}^{\frac{n-2}{n}}} \int_{\R^n} r^n \frac{(n+2-n r^2)}{(1+r^2)^{n+2}} dx $.
\end{itemize}
\end{thm}

\bigskip

\noindent
From the proof of Proposition 5.1 and Sections 4, 5 and 6 in \cite{MM1} we will need the estimates

\begin{enumerate}[label=(\roman{*})]
\item up to an error of order 
 $ 
O\bigg(\tau^{2}+\sum_{r}\frac{1}{\lambda_{r}^{4}}+\sum_{r\neq s}\varepsilon_{r, s}^{\frac{n+2}{n}}\bigg) $ there holds 
\begin{equation}\label{intergral_sum_of_bubble_nonlinear_evaluated}
\begin{split}
\int K & (\alpha^{i}\varphi_{i})^{p+1}d\mu_{g_{0}} \\
= &
\sum_{i}
\left(
\bar c_{0}\frac{K_{i}}{\lambda_{i}^{\theta}}\alpha_{i}^{p+1}
+
\bar c_{1}\frac{K_{i}}{\lambda_{i}^{\theta}}\alpha_{i}^{\frac{2n}{n-2}}\tau
+
\bar c_{2}\frac{\Delta K_{i}}{\lambda_{i}^{2+\theta}} \alpha_{i}^{\frac{2n}{n-2}} 
\right) \\
& \quad\quad\quad +
\bar d_{1}\sum_{i}\frac{K_{i}}{\lambda_{i}^{\theta}}\alpha_{i}^{\frac{2n}{n-2}}
\begin{pmatrix}
\frac{H_{i}}{\lambda_{i}^{3 }} \\
\frac{W_{i}\ln \lambda_{i}}{\lambda_{i}^{4}} \\
0 
\end{pmatrix} 
+
\bar b_{1}\sum_{i\neq j}\alpha_{i}^{\frac{n+2}{n-2}}\alpha_{j}
\frac{K_{i}}{\lambda_{i}^{\theta}}\varepsilon_{i, j}
\end{split}
\end{equation}
with 
 $ (\bar b_{1}=\frac{2n}{n-2}b_{1}) $ and
 $ 
\bar d_{1} 
=
\underset{\R^{n}}{\int} \frac{r^{n}dx}{(1+r^{2})^{n+1}}; 
 $ 
\item recalling \eqref{eq:eijm} we have
\begin{equation}\label{L_g_0_bubble_interaction}
\int \varphi_{i} L_{g_{0}} \varphi_{j}d\mu_{g_{0}}
=
\tilde b_{1}\varepsilon_{i, j}
+
O(\sum_{r\neq s}\frac{1}{\lambda_{r}^{4}}+\varepsilon_{r, s}^{\frac{n+2}{n}}), \; 
\quad 
\tilde b_{1}=4n(n-1)b_{1}; 
\end{equation}
\item up to an error 
 $ O(\tau^{2}+\frac{1}{\lambda_{i}^{4}}) $, there holds 
\begin{equation}\label{single_bubble_L_g_0_integral_expansion_exact}
\begin{split}
\int \frac{\varphi_{i} L_{g_{0}}\varphi_{i}}{4n(n-1)} d\mu_{g_{0}}
= &
\bar c_{0};
\end{split}
\end{equation}
 \item up to an error of order 
 $ 
O(\tau^{2}+\sum_{r}\frac{1}{\lambda_{r}^{4}}+\sum_{r\neq s}\varepsilon_{r, s}^{\frac{n+2}{n}})
 $ 
we have
\begin{equation}\label{r_alpha_delta_expansion} 
\begin{split}
\alpha^{i}\alpha^{j}\int \varphi_{i} L_{g_{0}}\varphi_{j}d\mu_{g_{0}}
= &
4n(n-1)\bar c_{0}
\sum_{i}
\alpha_{i}^{2}
+
\tilde b_{1}\sum_{i\neq j}\alpha_{i}\alpha_{j}\varepsilon_{i, j}. 
\end{split}
\end{equation}
\item If $ \varphi_{i} $ is as in \eqref{eq:bubbles}, then 
\begin{equation}
\begin{split}
\label{non_linear_v_part_interaction}
 \bigg\vert \int \varphi_{i}^{\frac{n+2}{n-2}}\nu d\mu_{g_{0}}\bigg\vert 
 \leq &
 \Vert v \Vert \bigg\Vert \frac{L_{g_{0}}\varphi_{i}}{4n(n-1)}-\varphi_{i}^{\frac{n+2}{n-2}}\bigg\Vert_{L_{g_{0}}^{\frac{2n}{n+2}}} 
\\ = &
 O
 \begin{pmatrix}
 \lambda_{i}^{-3} & \;\text{ for }\; n=5\\
 \ln^{\frac{2}{3}}\lambda_{i}\lambda_{i}^{-\frac{10}{3}} & \;\text{ for }\; n=6\\
 \lambda_{i}^{-4} &\;\text{ for }\; n\geq 7
 \end{pmatrix}
 \Vert v \Vert; 
\end{split}
 \end{equation}
\item up to an error
 $ O(\tau^{2}+\frac{1}{\lambda_{i}^{4}}) $ we have
with 
 $ 
\bar c_{2}=\frac{1}{2n}\underset{\R^{n}}{\int}\frac{r^{2}dx}{(1+r^{2})^{n}}; 
 $ 
 \begin{equation}\label{single_bubble_K_integral_expansion_exact}
\begin{split}
\int K \varphi_{i}^{p+1}d\mu_{g_{0}}
= &
\frac{\bar c_{0}K_{i}}{\lambda_{i}^{\theta}}
+
\bar c_{1}\frac{K_{i}\tau}{\lambda_{i}^{\theta}}
+
\bar c_{2}\frac{\Delta K_{i}}{\lambda_{i}^{2+\theta}}
+
\bar d_{1}K_{i}
\begin{pmatrix}
\frac{H_{i}}{\lambda_{i}^{3+\theta }} \\
\frac{W_{i}\ln \lambda_{i}}{\lambda_{i}^{4+\theta}} \\
0 
\end{pmatrix}
;
\end{split}
\end{equation} 
\item up to an error or order 
 $ 
O
(
\tau^{2}
+
\sum_{r\neq s} \frac{\vert \nabla K_{r}\vert^{2}}{\lambda_{r}^{2}}
+
\frac{1}{\lambda_{r}^{4}}
+
\varepsilon_{r, s}^{\frac{n+2}{n}}
)
 $ there holds
\begin{equation}\label{pure_functional_denominator_expanded}
\begin{split}
J_{\tau}(\alpha^{i} \varphi_{i})
= &
\frac
{\alpha^{i}\alpha^{j}\int \varphi_{i} L_{g_{0}}\varphi_{j}d\mu_{g_{0}}}
{(\int K(\sum_{i}\alpha_{i}\varphi_{i})^{p+1})^{\frac{2}{p+1}}} \\
= & 
\frac
{\alpha^{i}\alpha^{j}\int \varphi_{i} L_{g_{0}} \varphi_{j}d\mu_{g_{0}}}
{(\bar c_{0}\sum_{i}\frac{K_{i}}{\lambda_{i}^{\theta}}\alpha_{i}^{p+1})^{\frac{2}{p+1}}}
\Bigg(
1
-
\bar c_{1}\sum_{i}\frac{K_{i}}{\lambda_{i}^{\theta}}\frac{\alpha_{i}^{\frac{2n}{n-2}}}{\alpha_{K, \tau}^{\frac{2n}{n-2}}}\tau
\\
& \quad\quad\quad\quad\quad\quad\quad\quad\quad\quad\, -
\bar c_{2}\sum_{i}\frac{\Delta K_{i}}{\lambda_{i}^{2+\theta}} \frac{\alpha_{i}^{\frac{2n}{n-2}} }{\alpha_{K, \tau}^{\frac{2n}{n-2}}}
\\
& \quad\quad\quad\quad\quad\quad\quad\quad\quad\quad\, -
\bar d_{1}\sum_{i}\frac{K_{i}}{\lambda_{i}^{\theta}}
\begin{pmatrix}
\frac{H_{i}}{\lambda_{i}^{3 }} \\
\frac{W_{i}\ln \lambda_{i}}{\lambda_{i}^{4}} \\
0 
\end{pmatrix} 
\frac{\alpha_{i}^{\frac{2n}{n-2}}}{\alpha_{K, \tau}^{\frac{2n}{n-2}}}
\\
& \quad\quad\quad\quad\quad\quad\quad\quad\quad\quad\, -
\bar b_{1}\sum_{i\neq j}\frac{\alpha_{i}^{\frac{n+2}{n-2}}\alpha_{j}}{\alpha_{K, \tau}^{\frac{2n}{n-2}}}
\frac{K_{i}}{\lambda_{i}^{\theta}}\varepsilon_{i, j}
\Bigg);
\end{split}
\end{equation}
\item if $ \varepsilon_{i, j} $ is as in \eqref{eq:eijm}, then 
in case $ j<i\; \text{ or }\; 
d_{g_{0}}(a_{i}, a_{j})
\neq 
o(1) $ 
\begin{equation}\label{lambda_interaction_derivative_lower_bound}
\lambda_{j}\partial_{\lambda_{j}}\varepsilon_{i, j}
=
\frac{2-n}{2}
\varepsilon_{i, j}
+
O(\frac{1}{\lambda_{j}^{4}}+\varepsilon_{i, j}^{\frac{n+2}{n}})
.
\end{equation}
\end{enumerate}

\noindent Finally we derive one last
technical estimate. 
Recalling \eqref{eq:r}, 
from \eqref{r_alpha_delta_expansion} we have up to an error
 $ o(\frac{1}{\lambda^2}) $, 
 \begin{equation}\label{eq:barc0} 
\begin{split}
r_{\alpha^{i}\varphi_{i}}
= &
\alpha^{i}\alpha^{j}\int L_{g_{0}}\varphi_{i}\varphi_{j}d\mu_{g_{0}}
= 
4n(n-1)\bar c_{0}
\sum_{i}
\alpha_{i}^{2}
=
4n(n-1)\bar c_{0}\alpha^{2}
\end{split}
\end{equation}
with $ \bar c_{0}=\int_{\R^{n}}\frac{dx}{(1+r^{2})^{n}} $. 
From \eqref{intergral_sum_of_bubble_nonlinear_evaluated} instead we get 
\begin{equation*}
\begin{split}
\int K (\alpha^{i}\varphi_{i})^{p+1}d\mu_{g_{0}} 
= &
\sum_{i}
\left(
\bar c_{0}\frac{K_{i}}{\lambda_{i}^{\theta}}\alpha_{i}^{p+1}
+
\bar c_{1}\frac{K_{i}}{\lambda_{i}^{\theta}}\alpha_{i}^{\frac{2n}{n-2}}\tau
+
\bar c_{2}\frac{\Delta K_{i}}{\lambda_{i}^{2+\theta}} \alpha_{i}^{\frac{2n}{n-2}} 
\right) \\
= &
\bar c_{0} \alpha_{K, \theta}^{p+1}
+
\sum_{i}
\frac{K_{i}\alpha_{i}^{\frac{2n}{n-2}}}{\lambda_{i}^{\theta}}
\left(
\bar c_{1}\tau
+
\bar c_{2}\frac{\Delta K_{i}}{K_{i}\lambda_{i}^{2}}
\right)
\end{split}
\end{equation*} 
up to an error
 $ o(\frac{1}{\lambda^2}) $ 
and with constantsgiven by
\begin{equation}\label{eq:ovc1ovc2}
\bar c_{1}=\frac{2}{n-2}\int_{\R^{n}}\frac{\ln (1+r^{2})}{(1+r^{2})^{n}}dx, 
\; \quad \text{ and } \; \quad 
\bar c_{2}=\frac{1}{2n}\int_{\R^{n}}\frac{r^{2}}{(1+r^{2})^{n}}dx.
\end{equation}
Therefore 
\begin{equation*}
\begin{split}
\frac{r_{\alpha^{i}\varphi_{i}}}{(k_{\tau})_{\alpha^{i}\varphi_{i}}}
= & 
4n(n-1)\frac{\alpha^{2}}{\alpha_{K, \theta}^{p+1}}\\
& -
4n(n-1)
\frac{\alpha^{2}}{(\alpha_{K, \theta}^{p+1})^{2}}
\sum_{i}
\frac{K_{i}\alpha_{i}^{\frac{2n}{n-2}}}{\lambda_{i}^{\theta}}
\left(
\frac{\bar c_{1}}{\bar c_{0}}\tau
+
\frac{\bar c_{2}}{\bar c_{0}}\frac{\Delta K_{i}}{K_{i}\lambda_{i}^{2}}
\right) 
+
o(\frac{1}{\lambda^2})
\end{split}
\end{equation*}
and we conclude again from\eqref{def_refined_neighbourhood} that 
\begin{equation}\label{r/lambda_expansion}
\frac{r_{\alpha^{i}\varphi_{i}}}{(k_{\tau})_{\alpha^{i}\varphi_{i}}}
=
4n(n-1)\frac{\alpha^{2}}{\alpha_{K, \theta}^{p+1}}
(
1
-
(\frac{\bar c_{1}}{\bar c_{0}}-\frac{\tilde c_{1}}{\tilde c_{2}}\frac{\bar c_{2}}{ \bar c_{0}})\tau
)
+
o(\frac{1}{\lambda^2}).
\end{equation}

At last we display for the reader  conveniencethe equations where some dimensional 
constants appear. 

\begin{equation*}
\begin{array}{c||c|c|c|c|c|c}
&& \bar{\quad}& \hat{\quad} & \grave{\quad} & \tilde{\quad}&\check{\quad} \\ \hline
c_{0} && \eqref{eq:barc0} & & \eqref{constants_alpha_derivative} && \\ \hline
c_{1}& \text{Lemma}\; \ref{lem_interactions} & \eqref{eq:ovc1ovc2}& && \text{Theorem} \; \ref{lem_top_down_cascade}& \\ \hline 
c_{2} & \text{Lemma}\; \ref{lem_interactions} &\eqref{eq:ovc1ovc2} & &\eqref{constants_alpha_derivative} & \text{Theorem} \; \ref{lem_top_down_cascade}& \\ \hline 
c_{3} & \text{Lemma}\; \ref{lem_interactions} && \eqref{c_3} &&& \text{Theorem} \; \ref{lem_top_down_cascade} \\ \hline 
c_{4} &&&& && \text{Theorem} \; \ref{lem_top_down_cascade} \\ \hline 
d_{1} &&\eqref{intergral_sum_of_bubble_nonlinear_evaluated} && \eqref{constants_alpha_derivative} & \text{Theorem} \; \ref{lem_top_down_cascade}& \\ \hline
b_{1} & \text{Lemma}\; \ref{lem_interactions} &\eqref{intergral_sum_of_bubble_nonlinear_evaluated} && \eqref{constants_alpha_derivative}& \eqref{L_g_0_bubble_interaction} & \\ \hline 
b_{2} & \text{Lemma}\; \ref{lem_interactions} & &&&\text{Theorem} \; \ref{lem_top_down_cascade}&\\ \hline 
b_{3} & \text{Lemma}\; \ref{lem_interactions} &&&&&
\end{array}
\end{equation*}

\end{document}